\documentclass[12pt]{amsart}
\usepackage{verbatim,amscd,amssymb}
\usepackage{graphicx}
\usepackage{amsmath}
\numberwithin{equation}{subsection}
\usepackage[active]{srcltx}
\usepackage{enumerate}  % I added this to get i, ii, etc... in some enumeration (Arnaud)

\newtheorem{thm}{Theorem}[section]

\newtheorem{lem}[thm]{Lemma}

\newtheorem{cor}[thm]{Corollary}

\newtheorem*{thmA}{Theorem A}
\newtheorem*{thmB}{Theorem B}
\newtheorem*{thmC}{Theorem C}

\theoremstyle{definition}
\newtheorem{dfn}[thm]{Definition}

\newtheorem{rem}[thm]{Remark}

\def\Ac{\mathcal{A}}  
\def\C{\mathbb{C}}   \def\Fc{\mathcal{F}}
\def\Hc{\mathcal{H}}
 \def\al{\alpha}
  \def\Mc{\mathcal{M}}
\def\Pc{\mathcal{P}}    
\newcommand{\lam}{\mathcal{L}}

\def\R{\mathbb{R}}

\def\Dc{\mathcal{D}}

\def\Uc{\mathcal{U}} \def\Vc{\mathcal{V}} \def\Wc{\mathcal{W}}
\def\Xc{\mathcal{X}}

\def\Z{\mathbb{Z}}

\renewcommand\emptyset{\varnothing}
\newcommand{\sm}{\setminus}

\def\eps{\varepsilon}

\def\al{\alpha}
\def\be{\beta}

\def\ta{\theta}
\def\om{\omega}
\def\la{\lambda}
\def\si{\sigma}
\def\vp{\varphi}
\def\ol{\overline}
\def\re{\mathrm{re}}

\def\thu{\mathrm{TH}}

\def\ch{\mathrm{CH}}
\def\cu{\mathrm{CU}}
\def\uc{\mathbb{S}}
\def\bd{\mathrm{Bd}}
\def\le{\leqslant}
\def\ge{\geqslant}
\def\0{\emptyset}
\def\disk{\mathbb{D}}
\def\cdisk{\ol{\mathbb{D}}}
\def\phd{\mathrm{PHD}}

\def\lam{\mathcal L}

\def\cuc{\mathcal{CU}}

\usepackage{xspace}

\usepackage{xcolor}

\begin{document}
\date{May 25, 2019}
\title[Siegel capture polynomials in parameter spaces]
{Location of Siegel capture polynomials \\ in parameter spaces}

\author[A.~Blokh]{Alexander~Blokh}

\author[A.~Cheritat]{Arnaud~Ch\'eritat}

\author[L.~Oversteegen]{Lex~Oversteegen}

\author[V.~Timorin]{Vladlen~Timorin$^\star$}

\address[Alexander~Blokh and Lex~Oversteegen]
{Department of Mathematics\\ University of Alabama at Birmingham\\
Birmingham, AL 35294, USA}

\address[Vladlen~Timorin]
{Faculty of Mathematics\\
HSE University\\
6 Usacheva St., 119048 Moscow, Russia}

\address[Arnaud~Ch\'eritat]{Institut de Math\'ematiques\\Universit\'e Paul
Sabatier\\ 118 route de Narbonne, 31062 Toulouse, France}

\email[Alexander~Blokh]{ablokh@math.uab.edu}
\email[Arnaud~Ch\'eritat]{arnaud.cheritat@math.univ-toulouse.fr}
\email[Lex~Oversteegen]{overstee@uab.edu}
\email[Vladlen~Timorin]{vtimorin@hse.ru}

\subjclass[2010]{Primary 37F45; Secondary 37F10, 37F20, 37F50}

\keywords{Complex dynamics; Julia set;
laminations; Siegel capture polynomial}

\dedicatory{Dedicated to the memory of Anatole Katok}

\begin{abstract}
We study the set of cubic polynomials $f$ with a
marked fixed point. If $f$ has a Siegel disk at the marked fixed point,
and if this disk contains an eventual image of a critical point, we
call $f$ a \emph{IS-capture polynomial} (``IS'' stands for Invariant Siegel). We study the location of
IS-capture polynomials in the parameter space of all marked
cubic polynomials modulo affine conjugacy. In particular, we
show that any IS-capture polynomial is on the boundary of a unique bounded
hyperbolic component determined by the rational lamination of the map.
We also relate IS-capture polynomials to the cubic Principal Hyperbolic Domain and
its closure (by definition, the \emph{cubic Principal Hyperbolic
Domain} consists of cubic hyperbolic polynomials with Jordan curve
Julia sets).
\end{abstract}

\maketitle

\section{Introduction}
A complex polynomial $P$ of any degree is said to be \emph{hyperbolic}
if all of its critical points belong to the basins of
attracting or superattracting periodic cycles. The set of all
hyperbolic polynomials in any particular parameter space is open.
Components of this set are called \emph{hyperbolic components}. The
dynamics of hyperbolic complex polynomials is well understood.
According to the famous Fatou conjecture \cite{fat20}, hyperbolic
polynomials are dense in the parameter space of all complex
polynomials. This explains why hyperbolic components play a prominent
role in complex dynamics.

According to a general result of J. Milnor \cite{mil12}, every bounded hyperbolic component
in the moduli space of degree $d$ polynomials is an open topological cell of complex dimension $d-1$.
Hence it is fair to say that the structure of such hyperbolic domains is known.
However, in degrees greater than $2$, the same cannot be said about the closures of
hyperbolic components.
Arguably in the most important case of the cubic Principal Hyperbolic Domain $\phd_3$, the description of its boundary
has proved to be rather elusive.
For example, in a recent paper by L. Petersen and T. Lei \cite{lp09} it is shown  that the boundary of
$\phd_3$ has a  very intricate  ``fractal'' structure that is not fully understood.
Thus, understanding the boundaries of hyperbolic components in particular, understanding of the boundary of $\phd_3$,
is an important open problem.

Qualitative changes in the dynamics of polynomials take place on the
boundary of the connectedness locus. It is known that boundaries of
hyperbolic components are contained in the boundary of the entire
connectedness locus. This provides an additional incentive for studying
boundaries of hyperbolic components.

In our paper we consider these issues in the cubic case.
More precisely, we consider the parameter space of cubic polynomials with a marked fixed point.
The corresponding connectedness locus contains many complex analytic disks in its boundary.
A typical example is provided by \emph{IS-capture polynomials},  i.e., polynomials that have an
invariant Siegel domain around the marked fixed point and a critical point which is eventually mapped into it.
We prove that an IS-capture polynomial $f$ belongs to the boundary of a
\emph{unique} bounded hyperbolic component with the same \emph{rational lamination} as $f$.
We also prove that $f$ belongs to a complex analytic disk lying in the boundary of this hyperbolic component.
We give a combinatorial description of all hyperbolic components whose boundaries can contain IS-capture polynomials.

In \cite{bopt14} it was proven that all polynomials from $\phd_3$
satisfy some simple conditions. Our standing conjecture is that these
conditions are not only necessary but actually sufficient for a
polynomial to belong to the closure of $\phd_3$. Working in this
direction, we showed in \cite{bopt14b} that bounded components
complementary to the intersection of $\ol{\phd}_3$ with slices of the
parameter space can be either of so-called \emph{queer} type, or must
contain an IS-capture polynomial. The results of the present paper
together with those of \cite{bopt16a} allow one to rule out the case of
IS-capture polynomials.

Finally, we show that if a polynomial $P$ which does not belong to the
closure of $\phd_3$ and has multiplier $\lambda=e^{2\pi i \theta}$ at
its fixed point $w$, where $\theta$ is not a Brjuno number, then $w$ is
a Cremer fixed point of $P$. Thus, a counterexample to the Douady
conjecture in the cubic case \emph{must} be a polynomial from the
boundary of $\phd_3$.

\section{Detailed statement of the results}
\label{s:detailed}
We write $\C$ for the plane of complex numbers. The \emph{Julia set} of
a polynomial $f:\C\to \C$ is denoted by $J(f)$, and the \emph{filled
Julia set} of $f$ by $K(f)$. For quadratic polynomials, a crucial
object of study is the \emph{Mandelbrot set} $\Mc_2$.
Let $P_c(z)$ be a quadratic polynomial defined by the formula $P_c(z)=z^2+c$.
Clearly, $0$ is the only critical point of the polynomial $P_c$ in $\C$.
By definition, $c\in \Mc_2$ if the orbit of $0$ under $P_c$ is bounded.
Equivalently, $c\in \Mc_2$ if and only if the filled Julia set $K(P_c)$ is connected.
If $c\not\in \Mc_2$, then the set $K(P_c)$ is a Cantor set.

By \emph{classes} of polynomials we mean affine conjugacy classes. The
class of $f$ is denoted by $[f]$. Complex numbers $c$ are in one-to-one
correspondence with classes of quadratic polynomials. A higher-degree
analog of the set $\Mc_2$ is the {\em degree $d$ connectedness locus}
$\Mc_d$, i.e., the set of classes of degree $d$ polynomials $f$ all of
whose critical points do not escape or, equivalently, whose Julia set
$J(f)$ is connected.

The structure of the Mandelbrot set is described in the seminal work of
Thurston \cite{thu85} (see also \cite{DH}). In particular, \cite{thu85}
gives a full description of how distinct hyperbolic components of
$\Mc_2$ are located with respect to each other and what kind of
dynamics is exhibited by polynomials from their boundaries. However,
for degrees $d>2$ studying the set $\Mc_d$ has proven to be a difficult
task. The combinatorial structure of $\Mc_d$ remains elusive despite
some recent progress (see \cite{bopt17a} for the general case and
\cite{bopt17b, bopt17c} for the case of cubic polynomials with only
repelling periodic points).

The central and, arguably, the simplest part of the Mandelbrot set is
the \emph{$($quadratic$)$ Principal Hyperbolic Domain} denoted by
$\phd_2$. It is the set of all parameter values $c$ such that the
polynomial $P_c$ has an attracting fixed point. All these polynomials
have Jordan curve Julia sets. The closure $\ol\phd_2$ of $\phd_2$
consists of all parameter values $c$ such that $P_c$ has a
non-repelling fixed point. It is sometimes called the \emph{filled Main
Cardioid}. Its boundary $\bd(\phd_2)$ is a plane algebraic curve, a
cardioid called the \emph{Main Cardioid}. As follows from the
Douady--Hubbard parameter landing theorem and from the ``no ghost
limbs'' theorem by Yoccoz \cite{DH, Hu}, the Mandelbrot set itself can
be thought of as the union of $\ol{\phd}_2$ and \emph{limbs}, connected
components of $\Mc_2\sm\ol{\phd}_2$, parameterized by reduced rational
fractions $p/q\in (0,1)$.

It is natural to consider analogs of the Main Cardioid for higher
degree polynomials, in particular for cubic polynomials. This motivates
our interest to the boundary of the \emph{cubic Principal Hyperbolic
Domain} $\phd_3$ and to a closely related set, the so-called \emph{Main
Cubioid}, that was studied in a few recent papers (\cite{bopt14,
bopt14b, bopt16a, bopt16b}). In this framework an important task is to
describe whether polynomials with certain dynamical properties belong
to the boundary of the Main Cubioid. This is one of the problems
addressed in the present paper.

Let us now concentrate on cubic polynomials.
Let $\Fc$ be the space of polynomials $f_{\lambda,b}$ given by the formula
$$
f_{\lambda,b}(z)=\lambda z+b z^2+z^3,\quad \lambda\in \C,\quad b\in \C.
$$
The space $\Fc$ is adapted to studying polynomials with a marked fixed point.
Any such polynomial is affinely conjugate to one from $\Fc$ under a conjugacy sending the marked fixed point to $0$.
%Clearly, any affine conjugacy class of cubic polynomials has a
%representative in $\Fc$.
All polynomials $g\in \Fc$ have $0$ as a fixed point.
Let the \emph{$\la$-slice} $\Fc_\lambda$ of $\Fc$ be the space
of all polynomials $g\in\Fc$ with $g'(0)=\lambda$. It is well known
that two polynomials $f_{\la,b}$ and $f_{\la,b'}$ are conjugate by a
M\"obius transformation $M(z)$ that fixes $0$ if and only if $M(z)=\pm
z$ and $b'=\pm b$. We will deal with $f\in\Fc_\lambda$ for some
$\lambda$ and consider only perturbations of $f$ in $\Fc$. Set
$\Fc_{at}=\bigcup_{|\la|<1} \Fc_\la$ (the subscript $at$ stands for
\textbf{at}tracting).\footnote{The set $\Fc_{at}$ was denoted by $\Ac$
in \cite{bopt14b,bopt16b}. We adopt a more consistent notation in this
paper.} Let us emphasize that $\Fc_{at}$ is the family of polynomials
from $\Fc$ that have the point $0$ as an attracting fixed point. For
each $g\in \Fc_{at}$, let $A(g)$ be the immediate basin of attraction
of $0$. Denote by $\Fc_{nr}$ the set of all polynomials
$f=f_{\lambda,b}\in\Fc$ such that $0$ is
\textbf{n}on-\textbf{r}epelling for $f$ (so that $|\la|\le 1$).

Suppose that $a$ is a fixed point of a polynomial $f$ of any degree.
Assume that $f'(a)=e^{2\pi i \theta}$ where $\theta$ is irrational.
Then $a$ is said to be an \emph{irrationally indifferent} fixed point.
If $f$ is \emph{linearizable} (i.e., analytically conjugate to a
rotation) in a neighborhood of $a$, the point $a$ is called a
\emph{Siegel} fixed point. In this case the rotation in question is
well defined and is the rotation by $2\pi\theta$ so that $\theta$ is
called the \emph{rotation number}. Moreover, this is equivalent to the
existence of an \emph{orientation preserving topological} conjugacy
between $f$ in a neighborhood of $a$ and the rotation by $2\pi\theta$
of the unit disk. If $a$ is a Siegel fixed point, the biggest
neighborhood of $a$ on which $f$ is linearizable exists and is called
the \emph{Siegel disk} around $a$. If $f$ is not linearizable in any
neighborhood of $a$ then the point $a$ is called a \emph{Cremer} fixed
point.

\begin{dfn}[Siegel captures]
\label{d:compotypes} Suppose that a polynomial $f\in\Fc$ has a Siegel
disk $\Delta(f)$ around $0$. If a critical point of $f$ is
\emph{eventually} mapped to $\Delta(f)$, then this critical point is
denoted by $\mathrm{ca}(f)$ (here ``$\mathrm{ca}$'' stands for
``captured''), and $f$ is called an \emph{IS-capture polynomial}, or
simply an \emph{IS-capture} (here ``I'' stands for ``invariant''
and ``S'' stands for ``Siegel''). By \cite{man93}, there exists a
recurrent critical point $\mathrm{re}(f)$ of $f$ (here ``re'' stands
for ``recurrent'') whose limit set contains $\bd(\Delta(f))$. It
follows that the critical points $\mathrm{ca}(f)$ and $\mathrm{re}(f)$
are well-defined and distinct (evidently, $\mathrm{ca}(f)$ is not
recurrent).
\end{dfn}

\begin{rem}\label{r:0-siegel}
Generically, maps in the family $\Fc$ have three fixed points. Any of
these points, not only $0$, could have a Siegel disk around it that
captures a critical point. However, let us stress that we only speak of
IS-captures when $0$ is the Siegel fixed point whose Siegel disk captures a critical
point.
\end{rem}

In this paper, we study the location of IS-captures in $\Fc$ relative
to hyperbolic components. An important role here is played by the set
$\Pc^\circ$ of $\Fc$ of all hyperbolic polynomials $f\in \Fc$ such that
$f\in\Fc_{at}$ and $J(f)$ is a Jordan curve. Equivalently,
$f\in\Fc_{at}$ belongs to $\Pc^\circ$ if and only if $A(f)$, the
immediate basin of attraction of $0$, contains both critical points of
$f$. Evidently, $\Pc^\circ$ is open in $\Fc$. To see that
$\Pc^\circ$ is one hyperbolic component of $\Fc$, not only of
$\Fc_{at}$, observe that polynomials $f_{b,\la}=z^3+bz^2+\la z$ with
$|\lambda|=1$ are not hyperbolic and that by
Corollary~\ref{c:samecomp}, the set $\Pc^\circ$ is connected.

\begin{dfn}\label{d:princrit}
The set $\Pc^\circ$ is called the \emph{principal hyperbolic component} of $\Fc$.
We say that a hyperbolic polynomial $f\in\Fc_{at}$ is an
\emph{IA-capture polynomial} (IA stands for \emph{Invariant
Attracting}) if a critical point of $f$, denoted by $\om_2(f)$, is
eventually mapped to $A(f)$ but does not lie in $A(f)$ (then the
remaining critical point $\om_1(f)$ belongs to $A(f)$, and no critical
point of $f$ belongs to $J(f)$).
A hyperbolic component $\Uc$ of $\Fc$ is of \emph{IA-capture type} if $\Uc$ contains an IA-capture polynomial.
Hyperbolic components of IA-capture type will also be called \emph{IA-capture components}.
\end{dfn}

Similarly to Remark~\ref{r:0-siegel}, we emphasize that IA-capture
polynomials have $0$ as their attracting fixed point. Evidently, both
critical points $\om_1(f), \om_2(f)$ are well-defined for an IA-capture
polynomial $f$.
Observe also that, similarly to the above, the fact that polynomials
$f_{b,\la}=z^3+bz^2+\la z$ with $|\lambda|=1$ are not hyperbolic
implies that any hyperbolic component $\Uc$ of $\Fc$ of IA-capture type
is contained in $\Fc_{at}$. Thus, the principal hyperbolic component
$\Pc^\circ$ of $\Fc$ and the hyperbolic components of $\Fc$ of
IA-capture type are subsets of $\Fc_{at}$.

We also need the concepts of rational lamination and full lamination.
Denote by $\disk$ the open unit disk in the complex plane centered at
the origin and by $\uc$ the unit circle which is the boundary of
$\disk$. We will identify $\R/\Z$ with $\uc$ via $x\mapsto
e^{2\pi i x}$.

Let $f$ be a polynomial of degree greater than $1$ and connected Julia set.
In this case all external rays with rational arguments land.
Given two rational angles $\al,\be\in \R/\Z$, we declare $\al \sim_r \be$
 iff the landing points of the corresponding external rays coincide.
This defines an equivalence relation on $\mathbb{Q}/\Z$.
The equivalence classes are finite (see Theorems~\ref{t:a2} and \ref{t:a3} with references).
We then consider the collection $\lam^r_f$ of all edges of the convex hulls
of all equivalence classes and call it the \emph{rational lamination} of $f$.

If the Julia set $J(f)$ is locally connected, then all external rays
land. Given any two angles $\al,\be\in\R/\Z$ we declare that $\al\sim \be$ iff
the landing points coincide. This defines an equivalence relation
on $\R/\Z$, and in this case too the equivalence classes are finite
(see Theorems~\ref{t:a2} and \ref{t:a3} and Theorem 1.1 of
\cite{kiw02}). The collection of all edges of the convex hulls of all
classes is denoted $\lam_f$ and is called the \emph{(full) lamination}
of $f$. We will refer to the elements of $\lam_f$ as \emph{leaves}.

We include in each lamination the singletons $\{e^{2\pi i\al}\}$ and call
them \emph{degenerate leaves,} with $\al\in \mathbb{Q}/\Z$ for
$\lam^r_f$, resp.\ $\al\in \R/\Z$ for $\lam_f$. The set $\mathcal C$ of
all possible chords of the unit disk and singletons in
the unit circle is equipped with a natural topology that associates
to a chord $\ol{ab}$ of $\uc$ with endpoints $a$, $b\in\uc$ the pair $\{a,b\}$ of the symmetric product $\uc
\times \uc/(a,b)\sim(b,a)$.

Clearly, in the case when $J(f)$ is locally connected we have
$\lam^r_f\subset \lam_f$ and, since $\lam_f$ is closed (see
Section \ref{ss:pcuts}), $\ol{\lam^r_f}\subset \lam_f$. Contrary
to what one may expect, it is \emph{not always} true that
$\ol{\lam^r_f}=\lam_f$. A typical example is the case of a quadratic
polynomial $Q$ with invariant Siegel domain and locally connected Julia
set. Then $\lam^r_Q$ consists only of degenerate leaves and, therefore,
$\ol{\lam^r_Q}$ cannot be distinguished from the rational lamination of
$z^2$ (abusing the language we will call such a lamination the
\emph{empty} lamination). For IS-capture polynomials, we relate
rational and full laminations in Subsection \ref{ss:pcuts}. Recall that
a polynomial with connected Julia set that belongs to a hyperbolic
component has a locally connected Julia set and, hence, a well-defined lamination.

\begin{thmA}
  If $f\in\Fc$ is an IS-capture polynomial, then there is a unique
  bounded hyperbolic component $\Uc$ in $\Fc$, whose boundary contains $f$.
  Moreover, $\Uc\subset\Fc_{at}$, we have $\lam_P=\ol{\lam^r_P}$ for all $P\in\Uc$,
   and there are two possibilities:
  \begin{enumerate}
  \item the Julia set of $f$ contains no periodic cutpoints, then $\Uc=\Pc^\circ$;
  \item the Julia set of $f$ has a repelling periodic cutpoint, then $\Uc$ is of IA-capture type.
  \end{enumerate}
\end{thmA}

A polynomial is said to be \emph{$J$-stable with respect to a family of
polynomials} if its Julia set admits an equivariant holomorphic motion
over some neighborhood of the map in the given family \cite{lyu83,MSS}.
Say that $f\in \Fc_\la$ is \emph{$\la$-stable} if it is $J$-stable with
respect to $\Fc_\lambda$ with $\lambda=f'(0)$, otherwise $f$ is called
\emph{$\la$-unstable}.
A component of the set of $\la$-stable polynomials in $\Fc_\la$ is called an \emph{IS-capture component}
 if some (equivalently, all) polynomials from this component are IS-capture polynomials.
Thus IS-capture components are complex one-dimensional analytic disks in the two-dimensional space $\Fc$.
Every such disk is contained in a slice $\Fc_\la$ represented as a straight (complex) line in coordinates $(\la,b)$ of $\Fc$.

\begin{thmB}
  Every IS-capture polynomial belongs to some IS-capture component.
Every IS-capture component is contained in the boundary of a unique hyperbolic component $\Uc$ of $\Fc$.
Moreover, $\Uc=\Pc^\circ$ or $\Uc$ is of IA-capture type.
Conversely, let $\Uc$ be either an IA-capture component or $\Pc^\circ$.
Then the boundary of $\Uc$ contains uncountably many IS-capture components lying in $\Fc_\la$, where
 $\la=e^{2\pi i\ta}$, and $\ta$ runs through all Brjuno numbers in $\R/\Z$.
\end{thmB}

The first claim is contained in \cite[Theorem 5.3]{Z}.
A more precise formulation of the second part of Theorem B is contained in Theorem \ref{t:IS-BdP}.
We will apply Theorem A to the study of $\Pc$, the closure of $\Pc^\circ$ in $\Fc$.
The following are some properties of polynomials in $\Pc$.

\begin{thm}[\cite{bopt14}]\label{t:prophd}
If $f=f_{\la, b}\in \Pc$, then $|\la|\le 1$, the Julia set $J(f)$ is connected,
$f$ has no repelling periodic cutpoints in $J(f)$, and all its non-repelling
periodic points, except possibly $0$, have multiplier 1.
\end{thm}

These properties extend almost verbatim to the higher degree case \cite{bopt14}.
Theorem~\ref{t:prophd} motivates Definition~\ref{d:cubio}.

\begin{dfn}[\cite{bopt14}]\label{d:cubio}
Let $\mathcal{CU}$ be the family of cubic polynomials $f\in\bigcup_{|\la|\le 1}\Fc_\la$ such that $J(f)$ is connected,
$f$ has no repelling periodic cutpoints in $J(f)$, and all
its non-repelling periodic points, except possibly $0$, have multiplier 1.
The family $\mathcal{CU}$ is called the \emph{Main Cubioid} of $\Fc$.
\end{dfn}

Note that $\Pc^\circ$ and $\cuc$ are subsets of $\Fc$ that play a similar role to
 the principal hyperbolic component $\phd_3$ and the main cubioid $\cu$ in the (unmarked) moduli space of cubic polynomials.
However, the difference is that, when defining $\Pc^\circ$ and $\cuc$,
 we take into account the special role of the marked fixed point $0$ for polynomials in $\Fc$.
As a consequence, the sets $\Pc^\circ$ and $\cuc$ are not stable under arbitrary affine conjugacies.
By Theorem \ref{t:prophd}, this definition immediately implies that
\[\mathcal P \subset \mathcal{CU}.\]

For a compact set $X\subset \C$, define the \emph{topological hull
$\thu(X)$} of $X$ as the union of $X$ with all bounded components of
$\C\setminus X$. We write $\Pc_\la$ for the $\la$-slice of $\Pc$, i.e.,
for the set $\Pc\cap\Fc_\la$.

\begin{thmC}
IS-capture polynomials do not belong to $\mathcal{CU}\sm \Pc$.
If $\Wc$ is a component of $\thu(\Pc_\la)\sm \Pc_\la$ and $f\in \Wc$, then
the following holds.
\begin{enumerate}
\item Any such polynomial $f$ is $\la$-stable.
\item Critical points of $f$ are distinct and belong to $J(f)$.
\item The Julia set $J(f)$ has positive Lebesgue measure and carries an invariant
line field.
\end{enumerate}
\end{thmC}

Parts of Theorem B follow from \cite[Theorem 3.4]{Z}.

In Section \ref{s:cu}, we obtain corollaries of Theorem B that help
distinguish between Siegel and Cremer fixed points of a given multiplier.

\section{Rays and laminations}\label{ss:pcuts}

We will make use of the concepts of the full/rational \emph{lamination}
associated to a polynomial with connected Julia set. These concepts are
due to Thurston \cite{thu85} and Kiwi \cite{kiw97, kiw01, kiw04}. In
fact, in \cite{thu85} full laminations are defined independently of
polynomials as a combinatorial concept and are often studied in that
setting (see, e.g., \cite{bmov13}). Laminations are important tools of
combinatorial complex polynomial dynamics. Some of these tools are
applicable to polynomials of arbitrary degree, including those with
non-locally connected Julia sets. However, for the sake of brevity in
this paper we avoid unnecessary generality and define full lamination
only in the case when $P$ has a locally connected Julia set.

\subsection{Rays}

Studying periodic external rays of polynomials is a powerful tool in
complex dynamics. Given a polynomial $f$ with connected Julia set we
denote by $R_f(\al)$ the external ray of $f$ with argument $\al$.
(According to our convention, arguments of external rays are elements of $\R/\Z$ rather than $\R/2\pi\Z$.)
The arguments of external rays depend on the choice of a %linearizing
B\"ottcher coordinate near infinity. For an arbitrary cubic polynomial,
such coordinate is defined up to a sign, i.e., up to the involution
$z\mapsto -z$. However, for $f\in\Fc$, we can distinguish a linearizing
coordinate asymptotic to the identity. We assume that, whenever
$f\in\Fc$, the linearizing coordinate near infinity is chosen in this
way.

We begin this subsection by quoting known results concerning rational rays.

\begin{lem}[see, e.g., \cite{M}, Section 18]\label{l:land-ratio}
Let $f$ be a polynomial. All external rays of $f$ with (pre)periodic arguments
land.
The landing points eventually map to periodic points that are
parabolic or repelling. If $J(f)$ is connected then all rays landing at points that are eventually mapped
to parabolic or repelling periodic points have rational arguments.
\end{lem}

We call an external ray \emph{smooth} if it does not contain an escaping (pre)critical point.
Let us now state results on continuity of smooth periodic rays landing at repelling
periodic points. The next lemma can be found in \cite{GM} (Lemma B.1)
or \cite{DH} (Lecture VIII, Section II, Proposition 3).

\begin{lem}\label{l:rep}
Let $f$ be a polynomial, and $z$ be a repelling periodic point of $f$.
If a \emph{smooth periodic} ray $R_f(\theta)$ lands at $z$, then, for every polynomial $g$
sufficiently close to $f$, the ray $R_{g}(\theta)$ lands at a repelling
periodic point $w$ close to $z$, and $w$ depends holomorphically on $g$.
\end{lem}

A useful corollary of this lemma is stated below.

\begin{cor}[Lemma 4.7 \cite{bopt14b}]\label{c:converge}
Suppose that $h_n\to h$ is an infinite sequence of polynomials of
degree $d$ with connected Julia sets, and $\{\al, \be\}$ is a pair of periodic arguments such
that the external rays $R_{h_n}(\al)$, $R_{h_n}(\be)$ land at the same
repelling periodic point $x_n$ of $h_n$. If the external rays $R_h(\al)$, $R_h(\be)$
do not land at the same periodic point of $h$, then one of
these two rays must land at a parabolic point of $h$.
\end{cor}

Lemma 4.7 of \cite{bopt14b} is more general and includes (with provisions)
the case when Julia sets of polynomials $h_n$ are disconnected.

The following result is purely topological  and is based on local
behavior of polynomials at points of the plane. Given a polynomial $f$
with connected Julia set $J(f)$ and a point $z\in J(f)$, denote by
$A_z$ the set of arguments of rays landing at $z$. It is known
\cite{Hu} that $A_z$ is finite. Given a finite set $X\subset \uc$, the
points $a, b, c\in X$ are said to be \emph{consecutive} if the
positively oriented arcs $(a, b)$ and $(b, c)$ are disjoint from $X$
(observe that the order of points in this definition is essential).

\begin{thm}[cf Lemma 18.1 \cite{M}]\label{t:a1}
Let $f$ be a polynomial of degree $d>1$ whose Julia set $J(f)$ is
connected. (We do not assume that $J(f)$ is locally connected.) Let
$z\in J(f)$  be a point such that $A_z\ne \0$. Then $\sigma_d|_{A_z}$
is a $k$-to-$1$ map between $A_z$ and $A_{f(z)}$, and, if $z$ is
non-critical, then $k=1$. Moreover, there are two possibilities.
\begin{enumerate}
\item The set $\si_d(A_z)=A_{f(z)}$ is a singleton.
\item Given any three consecutive points $a, b, c$ in $A_z$, the
points $\si_d(a),$ $\si_d(b)$ and $\si_d(c)$ form a triple of consecutive points
in $A_{f(z)}$.
\end{enumerate}
\end{thm}

The next result is classical and has a proof using the Schwarz-Pick
metric in \cite{DH}. Recall that the (pre)periodic external rays are
exactly those whose arguments are rational.

\begin{thm}[Proposition 2, Section II, Lecture VIII \cite{DH}]\label{t:a2}
Let $f$ be a polynomial of degree greater than one with connected Julia
set. Then all rational external rays for $f$ land, and their landing points are
(pre)periodic points eventually mapped to repelling or parabolic
periodic points.
\end{thm}

The next theorem is a form of converse of Theorem~\ref{t:a2}. It is due to A. Douady.

\begin{thm}[Theorem I.A \cite{Hu}]\label{t:a3}
Let $f$ be a polynomial of degree $d>1$ whose Julia set $J(f)$ is connected.
Let $z\in J(f)$ be a repelling or parabolic periodic point. Then:
\begin{enumerate}[i.]
\item The point $z\in J(f)$ is the landing point of at least one periodic external ray.
\item Every external ray landing at $z$ is periodic.
\item All periodic external rays landing at $z$ have the same period.
\item There are finitely many external rays landing at $z$.
\end{enumerate}
\end{thm}

Once one proves the first claim, the others follow from it,
Theorem~\ref{t:a1} and properties of the $d$-tupling map.
In \cite{Hu} there is another proof, using the Yoccoz inequality.

The following nice theorem will not be used in its full strength; we
add it for the sake of completeness. A \emph{wandering point} in $J(f)$
is a point whose orbit is infinite: this is the opposite of being
(pre)periodic.

\begin{thm}[\cite{kiw02}]\label{t:a4}
Let $f$ be a polynomial of degree $d>1$ with locally connected Julia
set $J(f)$. Then there exists an integer $k=k(d)$ independent of $f$,
such that every wandering point $z\in J(f)$ can be the landing point of
at most $k$ external rays.
\end{thm}

\subsection{Full lamination}

For a (finite or infinite) set $A\subset \uc$, denote by
$\ch(A)$ its (closed Euclidian) convex hull.
A \emph{chord} $\ol{ab}$ between any two points $a,b\in\uc$ is
$\ch(\{a,b\})$ and contains the endpoints $a$ and $b$.
If $b=a$ the chord is called \emph{degenerate}.
Consider a closed set $A\subset \uc$ and its convex hull $\ch(A)$.
An \emph{edge} of
$\ch(A)$ is a closed straight segment $I$ connecting two points of
$\uc$ such that $I\subset \bd(\ch(A))$. Define the map $\si_d:\uc\to
\uc, \uc\subset \C,$ by $\si_d(s)=s^d$.
Then the
($\si_d$-)\emph{image} of a chord $\ol{ab}$ is by definition the chord
$\ol{\sigma_d(a)\sigma_d(b)}$. A ($\si_d$-)\emph{critical chord}
is a \emph{non-degenerate} chord whose endpoints have the same
$\si_d$-image.

Consider a locally connected Julia set $J(f)$ of a polynomial $f$ of
degree $d$. We can associate to $f$ an equivalence relation
$\sim_f$ on $\uc$ as follows. Two points $\al$ and $\be$ in $\uc$ are
$\sim_f$-equivalent if $R_f(\al)$ and $R_f(\be)$ land at the same
point. Then $J(f)$ is homeomorphic to $\uc/\sim_f$. By Theorems
\ref{t:a3} and \ref{t:a4}, any $\sim_f$-class is finite.

Recall that an equivalence relation $\sim$, or any binary relation, on
a set $X$ can be modeled by its \emph{graph}, the subset of pairs $(a,
b)\in X\times X$ such that $a\sim b$. If $f$ has a locally connected
Julia set, then all external rays land, and the landing point depends
continuously on the external angle, thus the graph of $\sim_f$ is a
closed subset of $\uc\times\uc$.

As mentioned in Section \ref{s:detailed}, the set $\mathcal C$ of all possible
degenerate and non-degenerate chords is naturally topologized by
associating to a chord $\ol{ab}$ the pair $\{a,b\}$ from
the symmetric square of $\uc$. Lemma~\ref{l:a6} follows from
the fact that the graph of $\sim_f$ is closed as well as the fact that
$\sim_f$-classes are finite and pairwise disjoint (we leave the proof
to the reader).

\begin{lem}\label{l:a6}
If $f$ has a locally connected Julia set, then $\lam_f$ is a closed
subset of the space of all chords of the unit disk with the above
topology.
\end{lem}

Moreover, $\si_d:\uc\to\uc$ descends to a map
$\si_d/\sim_f:\uc/\sim_f\to \uc/\sim_f$, and there exists a
homeomorphism $\phi:J(f)\to\uc/\sim_f$ conjugating $P|_{J(f)}$ with
$\si_d/\sim_f$. This shows that the equivalence relation $\sim_f$ gives
a precise model of the dynamics of $f|_{J(f)}$. However such a
precise model can be obtained only in the locally connected case.

With every $\sim_f$-class $G'$, we associate its convex hull
$G=\ch(G')$. The \emph{geodesic lamination $\lam_{\sim_f}=\lam_f$} is
defined as the set of edges of all such polygons $G$ together with all
singletons in $\uc$. We refer to $\lam_f$ as the \emph{(full) geodesic
lamination associated with $f$.} Elements of $\lam_f$ are called
\emph{leaves}. A leaf is \emph{degenerate} if it coincides with a point
in $\uc$; otherwise it is \emph{non-degenerate}. If $\ell = \ol{ab}$ is
a leaf, then, by Theorem~\ref{t:a1}, the chord
$\ol{\sigma_d(a)\sigma_d(b)}$ is again a (possibly degenerate) leaf. It
is denoted $\sigma_d(\ell)$. A \emph{critical leaf} is a leaf that is a
critical chord.

A \emph{gap} of $\lam_f$ is defined as the closure of a component of
$\disk\sm\bigcup\lam_f$. For any gap $G$ of $\lam_f$, we define $G'$ as
$G\cap\uc$ and $\si_d(G)$ as $\ch(\si_d(G'))$.  A gap $G$ is said to be
\emph{invariant} if $\si_d(G)=G$.
If $\lam_f$ has a gap $G$ such that $G'$ is infinite, then the union of the convex hulls of all $\sim_f$-classes does not equal $\cdisk$.

An equivalence relation $\sim$ on $\uc$, similar to $\sim_f$ above, can
be introduced with no references to polynomials.

\begin{dfn}[Laminational equivalence relations]\label{d:lam}

An equivalence relation $\sim$ on the unit circle $\uc$ is said to be
\emph{laminational} if:

\noindent (E1) the graph of $\sim$ is a closed subset in $\uc \times
\uc$;

\noindent (E2) convex hulls of distinct equivalence classes are
disjoint;

\noindent (E3) each equivalence class of $\sim$ is finite.
\end{dfn}

By an \emph{edge} of a $\sim$-class we mean an edge of its convex hull.

\begin{dfn}[Laminational equivalences and dynamics]\label{d:si-inv-lam}
A laminational equivalence relation $\sim$ is ($\si_d$-){\em
in\-va\-riant} if:

\noindent (D1) $\sim$ is {\em forward invariant}: for a class
$\mathbf{g}$, the set $\si_d(\mathbf{g})$ is a class too;

\noindent (D2) for any $\sim$-class $\mathbf{g}$, the map
$\tau=\si_d|_{\mathbf{g}}$ extends to $\uc$ as an orientation
preserving covering map $\hat \tau$ such that $\mathbf{g}$ is the full
preimage of $\tau(\mathbf{g})$ under the covering map $\hat \tau$.

\end{dfn}

To each laminational equivalence relation $\sim$ we associate the
corresponding \emph{geodesic lamination $\lam_\sim$} defined as the
collection of all edges of convex hulls of $\sim$-classes together with
all points of $\uc$. The terminology introduced for laminations
$\lam_f$ applies to laminations $\lam_\sim$ too. Abusing the language,
we will call the lamination all of whose leaves are singletons in $\uc$
the \emph{empty lamination}.

\subsection{Invariant gaps of cubic laminations}
Let $\lam_\sim$ be a cubic lamination. The \emph{degree} of a gap $G$
of $\lam_\sim$ is defined as the maximal number of disjoint critical
chords that can fit in $G$ and that are not on the boundary of $G$,
plus $1$, except for the case when $G$ is a triangle with critical
edges in which case the degree of $G$ is $3$. Recall that chords
include their endpoints, so by disjoint critical chords we mean chords
whose endpoints are distinct. Degree 2 (respectively, 3) gaps are
said to be \emph{quadratic} (respectively, \emph{cubic}).

By \cite{bopt14}, a quadratic $\si_3$-invariant gap $G$ has a unique
longest edge $M(G)$ called the \emph{major} (of $G$).
The major $M(G)$ can be either critical (then $G$ is said to be of
\emph{regular critical type}) or periodic (then $G$ is said to be of
\emph{periodic} type). For every edge $\ell=\ol{ab}$ of $G$, let
$H_\ell(G)$ be the arc of the unit circle $\uc$ with endpoints $a$ and $b$
and no points of $G$ in $H_\ell(G)$. Then the major
$M(G)$ can be singled out by the fact that the length of $H_{M(G)}(G)$
is greater than or equal to $1/3$ (here, the length of a circle arc is
normalized so that the total length of $\uc$ is equal to $1$).
We may summarize the above in the following theorem (recall that the
sets $\Fc_{at}$ and $\Pc^\circ$ were defined in Section \ref{s:detailed}; for
each $g\in \Fc_{at}$, we write $A(g)$ for the immediate basin of
attraction of $0$).

\begin{thm}[\cite{bopt16a}]\label{t:major}
Consider a polynomial $f\in \Fc_{at}\sm \Pc^\circ$ with locally
connected Julia set $J(f)$. Then the geodesic lamination $\lam_f$ has a
quadratic invariant gap $G$, and there are two possibilities.

\begin{enumerate}

\item The major $M(G)$ of $G$ is critical, the corresponding
    critical point of $f$ belongs to $\bd(A(f))$, and periodic
    cutpoints of $J(f)$ do not exist.

\item The major $M(G)$ of $G$ is periodic, and the corresponding
    point of $J(f)$ is a repelling or parabolic periodic
    cutpoint of $J(f)$.

\end{enumerate}

\end{thm}

Corollary~\ref{c:attracapt} easily follows.

\begin{cor}\label{c:attracapt} Suppose that $f\in \Fc_\la$, where $|\la|<1$, is
an IA-capture polynomial. Then $J(f)$ is locally connected, the
geodesic lamination $\lam_f$ has a quadratic invariant gap $G$ with
periodic major $M(G)$, the Julia set $J(f)$ contains a periodic
repelling cutpoint associated to $M(G)$, and %the polynomial
$f\in
\Fc_\la\sm \Pc$.
\end{cor}

\begin{proof} Since $f$ is hyperbolic, $J(f)$ is locally connected
so that Theorem~\ref{t:major} applies to $f$. Evidently, neither
critical point of $f$ belongs to $J(f)$. Hence case (1) of
Theorem~\ref{t:major} does not apply to $f$ while case (2) does apply.
The cutpoint cannot be parabolic for otherwise $f$ would not be
hyperbolic. This proves all claims of the corollary except for the last
one. To see that $f\in \Fc_\la\sm \Pc$ it remains to apply
Lemma~\ref{l:rep} which implies that small perturbations of $f$ will
have a periodic cutpoint in their Julia sets and, therefore, cannot
belong to $\Pc^\circ$.
\end{proof}

\subsection{Rational lamination}

A less informative but more universal concept is that of the
\emph{rational lamination $\lam^r_f$} associated to a polynomial $f$.
It was introduced by Kiwi (see \cite{kiw97, kiw01, kiw04}) and is based
upon the work of Goldberg and Milnor \cite{GM}.

\begin{dfn}[\cite{kiw97, kiw01, kiw04}]\label{d:ratiol}
Let $f$ be a polynomial of degree greater than $1$ with connected Julia set. Consider all its
external rays with rational arguments. If rays with distinct arguments
land at the same point in the plane, declare them equivalent so that
classes of equivalence are subsets of $\uc$. Then the convex hulls of
the thus defined classes of equivalence are pairwise disjoint. The set
formed by the edges of the convex hulls of all classes is called the
\emph{rational (geodesic) lamination} of $f$ and is denoted by
$\lam^r_f$.
\end{dfn}

Note that in the case $J(f)$ is locally connected, the equivalence
relation above is just the restriction of the equivalence relation
$\sim_f$ to the set of rational arguments.
Images of leaves of $\lam^r_f$ are leaves of $\lam^r_f$.
While the rational lamination $\lam^r_f$ contains valuable information about the dynamics of $f$,
 it does not define the polynomial, as we already noted in Section \ref{s:detailed}.

\begin{lem}\label{l:a5}
Let $f$ be a polynomial of degree $d\ge 2$ with connected Julia set. If
a chord is a limit of leaves $\ell_i\in \lam^r_f$ and one of its
endpoints is periodic, then its other endpoint is periodic of the same
period.
\end{lem}

\begin{proof}
We may assume that all leaves from the lemma are non-dege\-ne\-rate. By
Theorem \ref{t:a3}, we may assume that $\ell_i\in \lam_f^r$,
that $\ell_i\neq \ol{ab}$ for any $i$, and that $\ell_i\to
\ol{ab}$. We may also assume that $a$ is of period $m$ and $b$ is
either non-periodic or of period greater than $m$; then $\si_d^m(b)\ne
b$, and we may assume that $a\le \si_d^m(b)<b$ in the sense of the counterclockwise order on $\uc$.
A point in $\uc$ belongs to at most two non-degenerate leaves, hence we may assume that
$\ell_i=\ol{x_iy_i}$ with $x_i\to a, y_i\to b$ and $x_i \neq a$,
$y_i\neq b$. Then $a<x_i<b$, as otherwise $\si_d^m(\ell_i)$ crosses
$\ell_i$ as $\si_d^m$ at $a$ repels and does not change orientation.
Since distinct leaves $\ell_i$ do not cross, then $a<x_i<y_i<b$.
However, then, if $1\ll i\ll j$, the leaf $\si_d^m(\ell_j)$ crosses
$\ell_i$, a contradiction.
\end{proof}

If the Julia set of $f$ is locally connected, there is an alternative proof of
Lemma \ref{l:a5}: the endpoints of $\ell_i$ are equivalent under $\sim_f$,
and thus $a\sim_f b$ since $\sim_f$ is closed.
The claim then follows from Theorem \ref{t:a3}.

\begin{dfn}[\cite{bmov13}]\label{d:siblinv}
A collection of chords $\lam$ is \emph{sibling $\si_d$-invariant} provided
that:
\begin{enumerate}
\item for each $\ell\in\lam$, we have $\si_d(\ell)\in\lam$,
\item \label{2}for each $\ell\in\lam$ there exists $\ell_1\in\lam$
    so that $\si_d(\ell_1)=\ell$.
\item \label{3} for each $\ell\in\lam$ so that $\si_d(\ell)$
    is a non-degenerate leaf, there exist $d$ \textbf{disjoint} leaves
    $\ell_1$, $\dots$, $\ell_d$ in $\lam$ so that $\ell=\ell_1$ and
    $\si_d(\ell_i)=\si_d(\ell)$ for all $i=1$, $\dots$, $d$.
\end{enumerate}
\end{dfn}

We will use the following lemma.

\begin{lem}[\cite{kiw97,kiw01}]\label{l:ratio-sibl}
For a polynomial $f$ with connected Julia set the rational lamination $\lam^r_f$ is sibling invariant.
\end{lem}

\begin{proof}[Sketch of a proof]
Conditions (1) and (2) from Definition \ref{d:siblinv} are straightforward.
Assume that $\{z_1, \dots, z_k\}\subset J(f)$ is the
full preimage of a point $z\in J(f)$. Denote by $A_1$, $\dots$, $A_k$ and
$A$ the sets of arguments of external rays landing at $z_1$, $\dots$, $z_k$ and
$z$, respectively.
Condition (3) from Definition \ref{d:siblinv} is immediate for an
edge $\ell$ of the convex hull $\ch(A)$ except when some $z_i$'s are
critical. However then, too, one can find $r_i$ pairwise disjoint edges
of the convex hull $\ch(A_i)$ of $A_i$ where $r_i$ is the order of
$z_i$ so that each edge from this collection maps to $\ell$. This
completes the proof.
\end{proof}

The next lemma studies properties of infinite gaps of $\ol{\lam^r_f}$.

\begin{lem}\label{l:inf-lpr}
Edges of infinite gaps of $\ol{\lam^r_f}$ are (pre)periodic.
\end{lem}

\begin{proof}
Any leaf on the boundary of an infinite gap of any lamination $\lam$ is
eventually mapped to a critical leaf or to a periodic leaf, cf.
\cite[Lemma 4.5]{bopt17a}. However if an infinite gap $G$ has a
critical edge $\ell$, then by the properties of laminations this leaf
must be isolated in $\lam$.
If $\lam=\ol{\lam^r_f}$, then the above is only possible if $\ell\in \lam^r_f$ is (pre)periodic.
\end{proof}

We are ready to prove the next lemma.

\begin{lem}\label{l:ratiofull}
If $f$ is a polynomial of degree $d\ge 2$ with locally connected Julia
set and there is no bounded Fatou domain of $f$ whose boundary
contains a critical point with infinite orbit, then
$\ol{\lam^r_f}=\lam_f$.
\end{lem}

\begin{proof}
Recall that always $\ol{\lam^r_f}\subset \lam_f$. Suppose that
$\ol{\lam^r_f}\subsetneqq \lam_f$. By Lemma \ref{l:ratio-sibl}, the
collection $\lam^r_f$ is sibling invariant. Moreover, let $x$ and $y$
be rational arguments. By Theorems \ref{t:a2} and \ref{t:a3}, if $x\sim
y$ and $x$ is periodic for $\sigma_3$, then $y$ is periodic of the same
period. By Lemma \ref{l:a5} it follows that there are no critical
leaves in $\ol{\lam^r_f}$ with a periodic endpoint. Moreover, it
follows also that if $x\in \uc$ is periodic and $\ol{xy} \neq
\ol{xz}$ are leaves of $\ol{\lam^r_f}$, then $\si_d(\ol{xy})\ne
\si_d(\ol{xz})$. Sibling invariant collections of leaves with these
properties are called \emph{proper}; such collections as well as their
closures are studied in \cite{bmov13}. In particular, it follows from
Theorem 4.9 of \cite{bmov13} that $\ol{\lam^r_f}$ is a lamination
associated with an equivalence relation, say, $\approx$, on the unit
circle. This means that $\ol{\lam^r_f}$ is formed by the edges of the
convex hulls of all $\approx$-classes. Recall that $\lam_f$ is
generated by a specific equivalence relation on $\uc$ denoted by
$\sim_f$.

Now, by the assumption $\ol{\lam^r_f}\subsetneqq \lam_f$.
This implies that there is a gap $\widehat G$ of $\ol{\lam^r_f}$ that contains leaves of $\lam_f$ inside
(so that only the endpoints of these leaves belong to the boundary of $\widehat G$).
The gap $\widehat G$ cannot be finite because then all its vertices must be $\sim_f$-equivalent,
 and leaves of $\lam_f$ cannot intersect the interior of $\widehat G$.
Suppose that $\widehat G$ is infinite.
We claim that there are no infinite gaps $H$ of $\lam_f$ properly contained in $\widehat G$. Indeed, suppose otherwise.
Then an edge $\ell$ of $H$ must be contained in the interior of $\widehat G$ (except for its endpoints).
Observe that any edge of an infinite gap of any lamination is
either (pre)critical or (pre)periodic (cf.
\cite[Lemma 4.5]{bopt17a}). Since $\ell\in \lam_f\sm \ol{\lam^r_f}$, this implies that $\ell$ is
(pre)critical with infinite orbit, a contradiction with the assumption of the lemma. Thus,
all gaps of $\lam_f$ in $\widehat G$ are finite.

By \cite[Theorem 1.1]{kiw02}, all infinite gaps are (pre)periodic.
Hence for some $n$ the infinite gap $G=\si_d^n(\widehat G)$ is
periodic. By the previous paragraph all gaps of $\lam_f$ in $G$ are
finite. Then the quotient space $(G\cap\uc)/\sim_f$ is a so-called
\emph{dendrite}, which carries a self-map induced by  $\si_d^p$ where
$p$ is the minimal period of $G$. Theorem 7.2.7 from \cite{bfmot12}
implies that there are infinitely many periodic cutpoints in this
dendrite, hence $G$ contains leaves of $\lam^r_f$, a contradiction.
\end{proof}

\section{Preliminaries to Theorem A}\label{s:prel}

In this section, we list various preliminary results.
Some of them are well known and therefore given without proof.

\subsection{A perturbation lemma}\label{ss:pl}
Consider a sequence $\la_n\in\disk$ converging to $\lambda\in\uc$.
We say that $\la_n$ converges to $\la$ \emph{non-tangentially} if all $\la_n$ belong to a cone with the following properties.
The vertex of the cone is $\la$.
The axis of symmetry of the cone is the radius (radial line) through $\la$.
The angle between the edges of the cone and its axis of symmetry is less than $\pi/2$.
For an open set $U\subset\C$
and a holomorphic map $g:U\to\C$ with attracting fixed point $0$,
let $A(g)$ be the immediate basin of attraction of $0$ with
respect to $g$. Recall a part of Corollary 2 from \cite{BuPe},
based on ideas of \cite[Proposition 1, page 66]{yoc95}:

\begin{lem}[Corollary 2 of \cite{BuPe}]
 \label{l:perturb1}
 Suppose that $\la_n\in\disk$ converge non-tangentially to $\la\in\uc$.
 Let $U\subset\C$ be an open set, and $f:U\to\C$ be a holomorphic map with $f(0)=0$ and $f'(0)=\la$.
 Assume that $f$ has a Siegel disk $\Delta$ around $0$.
 If the sequence $f_n:U\to\C$ satisfies $f_n(0)=0$, $f'_n(0)=\la_n$, and for every compact subset $K\subset\Delta$
 $$
 \max_{z\in K}|f_n(z)-f(z)|=O(|\la-\la_n|),\quad n\to\infty,
 $$
 then any compact set $\widetilde K\subset\Delta$ is contained in $A(f_n)$ for $n$ large enough.
\end{lem}

We now go back to our family $\Fc$.
Below, we define some special perturbations of polynomials in $\Fc_{nr}$.
Let $f(z)=f_{\lambda,b}(z)=\lambda z+b z^2+z^3\in \Fc_{nr}$ so that $|\la|\le 1$.
Then denote by $f_\eps$ the polynomial
\begin{equation} \label{eq:feps}
f_{(1-\eps)\lambda, b}(z)=(1-\eps) \lambda z+ bz^2+z^3\in \Fc_{at},
\end{equation}
where $\eps>0$.
The following is an easy corollary of Lemma \ref{l:perturb1}.

\begin{cor}\label{c:perturb2}
If $f=\lambda z+b z^2+z^3$ has a Siegel disk $\Delta(f)$ around $0$,
then, for every compact set $\widetilde K\subset \Delta(f)$, there
exists $\delta(\widetilde K)>0$ such that every polynomial $f_\eps$ has
the property $\widetilde K\subset A(f_\eps)$ for any
$0<\eps<\delta(\widetilde K)$.
\end{cor}

\begin{proof}
Assume the contrary.
Then there exists a sequence $\eps_n\to 0$ with $\widetilde K\not\subset A(f_{\eps_n})$.
Set $\la_n=(1-\eps_n)\la$; then $\la_n$ converge to $\la$ non-tangentially.
To use Lemma \ref{l:perturb1}, observe that for a compact set $K\subset \Delta(f)$
$$
\max_{z\in K}|f_{\eps_n}(z)-f(z)|=O(|\la-\la_n|),\quad n\to\infty
$$
because the left-hand side equals $\eps_n\max_{z\in K}|z|$ while
$|\la-\la_n|=\eps_n$. This yields a
contradiction with Lemma \ref{l:perturb1} and proves the corollary.
\end{proof}

\subsection{Blaschke products}\label{ss:bp}

Here we deal with the dynamics of Blaschke products. As we do not
need Blaschke products of higher degrees and for the sake of simplicity we only
consider quadratic Blaschke products with fixed point $0$.
For a complex number $a$, we let ${\overline {a}}$ denote the complex conjugate of $a$.

\begin{dfn}[Blaschke products]\label{d:bp}
Let $b$ and $s$ be complex numbers such that $0<|b|<1$ and $|s|=1$.
Then the formula
\begin{equation} \label{eq:bp}
B_{b, s}(z)=sz{\frac {b-z}{1-{\overline {b}}z}}
\end{equation}
defines a \emph{quadratic Blaschke product with fixed point $0$}.
It is not hard to see that the Blaschke product \eqref{eq:bp} is conjugate by a rotation to a so-called
\emph{normalized quadratic Blaschke product} $Q_a$ of the form
\begin{equation} \label{eq:bp1}
Q_a(z)=z{\frac {a-z}{1-{\overline {a}}z}};
\end{equation}
for some complex number $a$ with $|a|<1$.
\end{dfn}

Our normalized Blaschke product $Q_a$ differs by a sign from the
traditional one in which the numerator is $z-a$, not $a-z$. It is well
known that $Q_a$ is a quadratic rational function that preserves
$\disk$, its complement $\C\sm \disk$, and the unit circle $\uc$.
Moreover,
\begin{equation}\label{eq:bpder}
Q'_a(z)={\frac {{\overline {a}}z^2-2z+a} {(1-{\overline {a}}z)^2}},
\end{equation}
which implies that $Q'_a(0)=a$;
an easy computation shows that the multiplier of the fixed point at $\infty$ is $\overline{a}$.
Thus, both $0$ and infinity are attracting fixed points of $Q_a$.
Set $\disk_r=\{|z|<r\}$; then, by the Schwarz Lemma (or directly),
we have $Q_a(\cdisk_r)\subset \disk_r$.
Similarly,
$|Q_a(z)|>|z|$ if $|z|>1$.
Hence the Julia set of $Q_a$ is $\uc$.
In fact, $Q_a$ is expanding on $\uc$, see \cite{tis00}.
For the sake of completeness, we now sketch the proof.
Assume that there exists a point $w\in \uc$ such that $|Q'_a(w)|\le 1$. Set
$s=\frac{w}{Q_a(w)}$ and consider $T(z)=sQ_a(z)$; then $|T'(w)|\le 1$
and $T(w)=w$ is a non-repelling fixed point of $T$. Since the Julia set of any
quadratic Blaschke product, in particular, of $T(z)$, is $\uc$, the point $w$ cannot be attracting.
By the Snail Lemma, $w$ must then be parabolic with multiplier $1$.
However, this contradicts the fact that points near $\uc$ are attracted to $0$ and $\infty$
(and repelled away from $\uc$) by $T$.

Solving the
quadratic equation ${\overline {a}}z^2-2z+a=0$, we see that the critical
points of $Q_a(z)$ are given by
$$
c_a=\frac{1 - \sqrt{1-|a|^2}}{{\overline {a}}}\quad
\textrm{and}
\quad
d_a=\frac{1 + \sqrt{1-|a|^2}}{{\overline {a}}}
$$
It is easy to see, that
\begin{equation}\label{eq:bpcrpt}
c_a=\frac{1 - \sqrt{1-|a|^2}}{{\overline
{a}}}=a\frac{1-\sqrt{1-|a|^2}}{|a|^2}=\frac{a}{1+\sqrt{1-|a|^2}}
\end{equation}
is the unique critical point of $Q_a$ that belongs to $\disk$.
Also, by \eqref{eq:bpcrpt} $a$ and $c_a$ belong to the same radial segment of $\cdisk$ so
that $c_a$ is located between $0$ and $a$.
Observe that if
$a\to s\in \uc$, then $c_a\to s$ too. To describe the limit behavior of
the entire orbit of $c_a$ as $a\to s\in \uc$, we need Lemma
\ref{l:limit-onk}. For a complex number $w$, set $\mathrm{R}_w(z)=wz$.

\begin{lem}\label{l:limit-onk}
Suppose that $s\in \uc$ and $K\subset \C\sm \{s\}$ is a compact set.
Then the maps $Q_a$ converge to $\mathrm{R}_s$ uniformly on $K$ as $a\to s$.
\end{lem}

\begin{proof}
Since $|s|=1$, we have $s\ol{s}=1$. Therefore $s-z=s-s\ol{s}z=s(1-\ol{s}z)$. Dividing on both sides by
$1-\ol{s}z$, we see that $\frac{s-z}{1-\ol{s}z}=s$ for all $z\ne \frac{1}{\ol{s}}=s$. Since
$K\subset \C\sm \{s\}$ is a compact set, standard continuity arguments imply the conclusions
of the lemma.
\end{proof}

This does not yet yield the limit behavior of the orbit of $c_a$ as $a\to s\in \uc$ as then
$c_a\to s$ too, and Lemma~\ref{l:limit-onk} does not apply.

\begin{lem}\label{l:bpcror}
Suppose that $s=e^{2\pi i \ta}$, where $\ta$ is irrational. Let $\eps$
be a positive real number and $m$ be a positive integer. Then there
exists $\delta>0$ such that for any $a\in \disk$ with $|s-a|<\delta$ we
have $|Q^i_a(c_a)|>1-\eps$ for all $i=0, 1, \dots, m$.
\end{lem}

In other words, if $a=Q'_a(0)$ is close to $s$, then the orbit of $c_a$
stays close to the unit circle for any given period of time. The
conclusions of the lemma are sensitive with respect to the point whose
trajectory we consider. For example, $Q_a(a)=0$ so that the orbit of
$a$ under $Q_a$ is $(a, 0, 0, \dots)$ and, thus, the limit behavior of
the orbits of $a$ and of $c_a$ are very different even though both $a$
and $c_a$ converge to $s=e^{2\pi i \ta}$.

\begin{proof}
We will use the following notation and terminology.
Given a \emph{small} arc $T\subset \uc$ with endpoints of arguments $\al$ and $\be$,
denote by $U_T$ a ``polar rectangle''
built upon $T$ with vertices (in polar coordinates) given by $(1-|T|,
\al),$ $(1+|T|, \al),$ $(1+|T|, \be),$ $(1-|T|, \be).$

Simple computations show that

\begin{equation}\label{eq:bpcrim}
Q_a(c_a)=\frac{(1-\sqrt{1-|a|^2})^2}{\overline{a}^2}=c^2_a
\end{equation}

Since $\ta$ is irrational, there exists a closed arc
$I\subset \uc$ symmetric with respect to $s$ such that $I,$
$\mathrm{R}_s(I),$ $\mathrm{R}^2_s(I),$ $\dots,$ $\mathrm{R}^m_s(I)$ are
pairwise disjoint circle arcs. By Lemma~\ref{l:limit-onk}, we can choose
a small arc $T\subset \mathrm{R}_s(I)$ centered at $s^2$ such that for all
$a$ sufficiently close to $s$ we have that $Q^i_a(U_T)\subset
U_{\mathrm{R}^{i+1}_s(I)}$ for all $i=0, \dots, m-1$.
We can then choose a small neighborhood $W$ of $s$ so that
$\zeta^2\subset U_T$ provided that $\zeta\in W$;
by \eqref{eq:bpcrpt} and \eqref{eq:bpcrim} this implies that for any $a$ sufficiently close to
$s$ we have $c_a\in W$ and $Q_a^j(c_a)\in U_{\mathrm{R}^j_s(I)}$ for
every $j=1, \dots, m$.
\end{proof}

\subsection{Modulus}\label{ss:mod}
The notion of the modulus of an annulus is widely used in complex dynamics.
We need several well known facts concerning annuli and their moduli.

\begin{dfn}\label{d:annul} A \emph{round} annulus $A(r, R)\subset \C$
is an open annulus formed by two concentric circles of radii $r<R$. A
\emph{topological} annulus $U\sm K$ is formed by a simply connected
domain $U\subset \C$ and a \emph{non-separating} (i.e., such
that $\C\sm K$ is connected) continuum $K\subset U$.
If $K$ is not a singleton and $U\ne\C$, then we will call $U\sm K$ \emph{non-degenerate}.
\end{dfn}

An important fact concerning annuli is the following well known version
of the Riemann Mapping Theorem.

\begin{thm}[see, e.g., Theorem 10, Section 5, Chapter 6 \cite{ahl79}]
Any non-degenerate annulus is conformally equivalent to a
non-degenerate round annulus.
\end{thm}

Let us also mention classical Schottky's Theorem.

\begin{thm}[\cite{sch877}]\label{t:sch}
Two round annuli $A(r, R)$ and $A(r', R')$ are conformally equivalent
if and only if $\frac{R}{r}=\frac{R'}{r'}$.
\end{thm}

We are ready to define the modulus.

\begin{dfn}\label{d:mod} For a round annulus $A(r, R)$
its \emph{modulus} $m(A(r, R))$ is defined as the number
$\frac{\ln(R)-\ln(r)}{2\pi}$. Given a topological annulus $\widehat A$
that is conformally equivalent to the round annulus $A=A(r, R)$, we set
$m(\widehat A)=m(A)=\frac{\ln(R)-\ln(r)}{2\pi}$.
\end{dfn}

Observe that by Schottky's Theorem the modulus of an annulus is well
defined and invariant under conformal equivalence. We need
Theorem~\ref{t:annuprop} concerning this concept; below $d(X, Y)$
denotes the infimum of the distance between points $x\in X$ and $y\in
Y$ for sets $X, Y\subset \C$.

\begin{thm}\label{t:annuprop}
Suppose that $A\subset A'$ are two annuli such that $A$ is not null-homotopic in $A'$.
Then $m(A)\le m(A')$.
Moreover, there exists a function $\psi:\R_{>0}\to \R_{>0}$ such that $d(K, \uc)\ge \psi(m(\disk\sm K))$
  for any non-separating continuum $K\subset\disk$.
\end{thm}

The first part of Theorem~\ref{t:annuprop} is well known and can be found in various textbooks;
 the second part easily follows, e.g., from \cite[Theorem 2.4]{Mc} or from \cite[Problem I of Section A, Chapter III]{Ahl06}.

\subsection{Hyperbolic components}
We will make use of the following result \cite[Corollary 2.10]{McS}:

\begin{lem}
  \label{l:McS}
  Let $f$ be a hyperbolic rational function.
  Then the set $[f]_{top}$ of rational functions topologically conjugate to $f$ coincides with
  the set of rational functions qc-conjugate to $f$ and is connected.
\end{lem}

Suppose now that $f$ and $g$ are hyperbolic polynomials in $\Fc$ with
connected Julia sets. Recall that then $J(f)$, $J(g)$ are locally
connected. A \emph{critical orbit relation} for $f$ is a constraint of
the form $f^n(c)=f^m(d)$, $m\ne n$, where $c$ and $d$ are critical
points of $f$, not necessarily different.
As in Section \ref{ss:pcuts}, we can associate geodesic laminations
$\lam_f$ and $\lam_g$ with $f$ and $g$, respectively.

\begin{lem}
  \label{l:topconj}
  Let $f$ and $g$ be two degree $d>1$ hyperbolic polynomials with connected Julia sets such that $\lam_f=\lam_g$.
If $f$ and $g$ have no critical orbit relations, then $f$ and $g$ are topologically conjugate.
\end{lem}

See \cite{McS} for very similar statements.
The same methods prove Lemma \ref{l:topconj}.
It follows that $g\in [f]_{top}$.
Note however that, in the cubic case, the intersection of $[f]_{top}$ with $\Fc$ may be disconnected.

\begin{cor}
  \label{c:samecomp}
If polynomials $f$ and $g$ belong to the same bounded hyperbolic component of $\Fc$, then $\lam_{f}=\lam_{g}$.
On the other hand, suppose that $f$, $g\in\Fc_{at}$ are hyperbolic polynomials with connected Julia sets such that $\lam_f=\lam_g=\lam$.
If $f$ and $g$ have no attracting fixed points except $0$, then $f$, $g$ belong to the same hyperbolic component of $\Fc$.
\end{cor}

\begin{proof}
The first claim is a variation of a well-known property of hyperbolic components; it is left to the reader.
To prove the rest,
we may assume that neither $f$ nor $g$ has critical orbit relations.
Indeed, otherwise we can slightly perturb $f$ and $g$ within their hyperbolic components of $\Fc$ so that
the perturbed maps have no critical orbit relations.
Then $f$ and $g$ are topologically conjugate by Lemma \ref{l:topconj}.
Suppose that $f=f_{\la_f, b_f}=z^3+b_fz^2+\la_f z$ and
$g=g_{\la_g, b_g}=z^3+b_gz^2+\la_g z$.

By Lemma \ref{l:McS}, there is a continuous family $f_t$, $t\in [0,1]$
of cubic rational functions qc-conjugate to $f$ such that $f_0=f$ and $f_1=g$.
Indeed, a qc-conjugacy between $f$ and $g$ takes the standard complex structure on the dynamical plane of $g$ to some invariant qc-structure on the dynamical plane of $f$.
The latter is represented by a Beltrami differential $\nu$.
Considering the family of Beltrami differentials $\nu_t=t\nu$ and using the Ahlfors--Bers theorem, we obtain a family $f_t$
with the desired properties. Observe that all rational functions $f_t$ are hyperbolic.

Let $M_t$ be a complex affine transformation such that $h_t=M_t\circ f_t\circ M_t^{-1}\in\Fc$.
Since $[0,1]$ is simply connected, we may choose $M_t$ to depend continuously on $t$ and so that $M_0=id$.
Let $\Uc$ be the hyperbolic component of $\Fc$ containing $f$.
Then $h_t\in \Uc$ for all $t$ by continuity; in particular, $h_1\in \Uc$.
On the other hand, $h_1=M_1\circ g\circ M_1^{-1}\in\Fc$ and $g$ are affinely conjugate.
This implies that either $h_1=g$ or $h_1=z^3-b_g z^2+\la_g z$.
In the former case, we are done.
In the latter case, observe that $h_1$ and $g$ have the same linearizing coordinate near infinity
(this follows from the fact that $z\mapsto z^3$ commutes with the involution $z\mapsto -z$) while
the orbits of $g$ are obtained from the orbits of $h_1$ by $z\mapsto -z$.
Therefore, the geodesic lamination of $g$ differs from the geodesic lamination of $h_1$ by a half-turn.

On the other hand, by our construction $\lam$ coincides with the geodesic lamination of $h_1$.
Thus, $\lam$ is invariant with respect to the rotation by $180$ degrees
about the center of the unit disk. % (i.e., $\lam$ is
%symmetric with respect to the center of the unit disk).
Then, by \cite{bopt16a}, the major of an invariant quadratic gap $G$ in $\lam$ corresponding to
the basin of immediate attraction of $0$ (of either $f$ or $g$) is $\ol{0\frac 12}$.
This implies that there are two invariant attracting domains of $g$ (or $f$), corresponding to $G$
and the $180$-degree rotation of $G$ with respect to the center of the unit disk.
A contradiction with the assumption that $g$ (and $f$) has only one attracting fixed point.
The statement now follows.
\end{proof}

\section{Proof of Theorem A}\label{s:mt}

Let $\Wc$ be a component of $\thu(\Pc_\la)\sm\Pc_\la$, where $|\la|\le
1$. It is called a \emph{queer domain} (or is said to be of
\emph{queer} type) if there exists a polynomial $f\in \Wc$ so that all
of its critical points are in $J(f)$. Polynomials from such $\Wc$ are
also said to be of \emph{queer type}. Observe that IS-polynomials and
polynomials of queer type have connected Julia sets. If $f$ is an
IS-polynomial, then $\mathrm{ca}(f)$ is a critical point of $f$ that
does not belong to $J(f)$, hence $f$ is \emph{not} a polynomial of
queer type.

The following theorem relies on \cite[Theorem 3.4]{Z}, where the most
difficult case is worked out.

\begin{thm}[\cite{bopt14b}]\label{t:nosiegel}
Let $\Wc$ be a component of $\thu(\Pc_\la)\sm\Pc_\la$ of queer type.
Then, for any polynomial $f\in\Wc$, the Julia set $J(f)$ has positive
Lebesgue measure and carries an invariant line field.
\end{thm}

Properties of polynomials from $\Pc$ listed in Theorem~\ref{t:prophd}
are inherited by polynomials from the topological hulls
$\thu(\Pc_\la)$.

\begin{thm}[\cite{bopt14}]\label{t:extendclo}
Suppose that $|\la|\le 1$.
We have
\[\thu(\Pc_\la)\subset \mathcal{CU}.\]
Moreover, all components of the set $\thu(\Pc_\la)\sm \Pc_\la$,
where $|\la|\le 1$, consist of $\la$-stable polynomials.
\end{thm}

In \cite{bopt14b}, we consider components of the set $\thu(\Pc_\la)\sm \Pc_\la$, where $|\la|\le 1$.
Let us describe some results of \cite{ bopt14b, bopt16b}.
A cubic polynomial $f\in\Fc_\la\sm\Pc=\Fc_\la\sm\Pc_\la$ with $|\la|\le 1$ is said to be \emph{potentially renormalizable}.
Perturbing a potentially renormalizable polynomial $f\in \Fc_{nr}\sm \Pc$ to a polynomial $g\in \Fc_{at}$,
we see that $g|_{A(g)}$ is two-to-one (otherwise %$[f]\in \ol\phd_3$).
$f\in \Pc$) and, hence, $g$ has two distinct critical points. By Lemma~\ref{l:2crpts} below this
property is inherited by a potentially renormalizable polynomial.

\begin{lem}[\cite{bopt14b}] \label{l:2crpts}
A potentially renormalizable polynomial has two distinct critical points.
\end{lem}

A critical point $c$ of a potentially renormalizable polynomial $f$ is said to be \emph{principal} if there is
a neighborhood $\Uc$ of $f$ in $\Fc$ and a holomorphic function
$\omega_1:\Uc\to\C$ defined on $\Uc$ such that $c=\omega_1(f)$, and,
for every $g\in\Uc\cap\Fc_{at}$, the point $\omega_1(g)$ is the critical
point of $g$ contained in $A(g)$.

\begin{thm}[\cite{bopt14b}] \label{t:princ}
A potentially renormalizable polynomial has a unique principal critical point.
\end{thm}

By Theorem~\ref{t:princ}, if $f\in \Fc_{nr}$ is potentially
renormalizable, then the point $\om_1(f)$ is well-defined; let the
other critical point of $f$ be $\om_2(f)$. It is easy to see that
$\om_1(f)\in K(f)$. It is proven in \cite{bopt16a} (see
Theorem~\ref{t:major} and Corollary~\ref{c:attracapt} where the
appropriate results of \cite{bopt16a} are summarized) that an
IA-capture polynomial $g$ has a repelling periodic cutpoint of the
Julia set $J(g)$. Hence an IA-capture polynomial $g$ is not in
$\mathcal{CU}$, thus not in $\mathcal{P}$, i.e., it is potentially
renormalizable, and the notation for its critical points $\om_1(g)$,
$\om_2(g)$, introduced in Definition~\ref{d:princrit}, is consistent
with the just introduced notation for all potentially renormalizable
polynomials.

Recall that, by Theorem~\ref{t:extendclo}, all polynomials in a
component $\Wc$ of $\thu(\Pc_\la)\sm\Pc_\la$ are conjugate on their
Julia set. Moreover, if some polynomial in $\Wc$ is an IS-capture,
then it is easy to see that so are all polynomials in $\Wc$. This
inspires the following definition. Let $\Wc$ be a component of
$\lambda$-stable polynomials, where $|\la|\le 1$. Then $\Wc$ is said to
be of \emph{IS-capture type} if any $f\in \Wc$ is an IS-capture polynomial.
We also say in this case that $\Wc$ is an \emph{IS-capture component}.
It is easy to construct examples of IS-captures in $\Fc_\la\sm\thu(\Pc_\la)$.

\begin{thm}[\cite{bopt14b}]\label{t:sie-quee}
Let $\Wc$ be a component of\, $\thu(\Pc_\la)\sm \Pc_\la$, where $|\la|\le 1$.
Then $\Wc$ is either of IS-\emph{capture type} or of \emph{queer} type.
\end{thm}

By Theorem C, which we prove later on, the first possibility listed in Theorem~\ref{t:sie-quee} is impossible.

Let us introduce notation, which will be used in the rest of this
section. Let $f$ be an IS-capture polynomial. As before, we write
$\Delta(f)$ for the Siegel disk of $f$. Recall that $f$ has two
distinct critical points $\mathrm{re}(f)$ and $\mathrm{ca}(f)$ (``re''
from ``recurrent'' and ``ca'' from ``captured''), see Definition
\ref{d:compotypes}. The point $\mathrm{re}(f)$ is recurrent, and the
closure of its forward orbit includes $\bd(\Delta(f))$. The point
$\mathrm{ca}(f)$ is eventually mapped to $\Delta(f)$. Let $m_f>0$ be
the smallest positive integer for which we have
$f^{m_f}(\mathrm{ca}(f))\in \Delta(f)$. Observe that, given
sufficiently small $\eps>0$, for all polynomials $g$ close enough to
$f$, there exist a unique critical point $\mathrm{re}(g)$ of $g$ that
is $\eps$-close to $\mathrm{re}(f)$ and a unique critical point
$\mathrm{ca}(g)$ of $g$ that is $\eps$-close to $\mathrm{ca}(f)$.
Notice that the functions $\mathrm{re}(g)$ and $\mathrm{ca}(g)$ are
holomorphic functions of the coefficients of $g$. However $\re(g)$ is
not necessarily recurrent, and $g$ may not have a Siegel invariant
domain.

Lemma~\ref{l:zeta-out} is based on special perturbations
\eqref{eq:feps} introduced right before Corollary \ref{c:perturb2}.
Namely, recall that for $f(z)=f_{\lambda,b}(z)=\lambda z+b z^2+z^3\in
\Fc_{nr}$ with $|\la|\le 1$ we denote by $f_\eps$ the polynomial
$f_{(1-\eps)\lambda, b}(z)=(1-\eps) \lambda z+ bz^2+z^3\in \Fc_{at}$,
where $\eps>0$ is a small positive number.

\begin{lem}\label{l:zeta-out}
Suppose that $f$ is an IS-capture polynomial.
Then, for sufficiently small $\eps>0$, we have $\mathrm{re}(f_\eps)\in A(f_\eps)$.
In particular, if $f_\eps\notin \Pc^\circ$, then $\mathrm{ca}(f_\eps)\notin A(f_\eps)$.
\end{lem}

\begin{proof}
Set $f=f_{\la, b}$. Then $\la=e^{2\pi i \ta}$, where $\ta$ is irrational.
Take a closed Jordan disk $K$ and an open Jordan disk $U$ such that
$$
0\in K\subset U\subset\ol U\subset \Delta(f).
$$
We may assume that $f^{m_f}(\mathrm{ca}(f))$ lies in the interior of $K$.

Observe that if $f_{\eps}\in \Pc^\circ$ then $\mathrm{re}(f_\eps)\in
A(f_\eps)$ as desired. In particular, if for sufficiently small
$\eps>0$ we have that $f_{\eps}\in \Pc^\circ$, then we are done. Thus
we need to consider the case when there are positive values of $\eps$
arbitrarily close to $0$ and such that $f_{\eps}\notin \Pc^\circ$. We
need to show that $\mathrm{re}(f_\eps)\in A(f_\eps)$ for all these values
of $\eps$. Observe that in any case at least one critical point must
belong to $A(f_\eps)$ for all $\eps>0$. Hence, if
$\mathrm{ca}(f_\eps)\notin A(f_\eps)$ for some $\eps>0$, then
$\mathrm{re}(f_\eps)\in A(f_\eps)$ for this $\eps$ as desired. Thus, to
prove the lemma it would suffice to prove the following claim.

\medskip

\noindent \textbf{Claim}. \emph{For sufficiently small $\eps>0$, if
$f_{\eps}\notin \Pc^\circ$ then $\mathrm{ca}(f_\eps)\notin A(f_\eps)$.}

\medskip

\noindent \emph{Proof of the Claim}. Suppose that there are positive
values of $\eps$ arbitrarily close to $0$ and such that $f_{\eps}\notin
\Pc^\circ$. Moreover, suppose by way of contradiction that the Claim
fails. Then there exists a sequence $\eps_n\to 0$ with
$f_{\eps_n}\notin \Pc^\circ$ and $\mathrm{ca}(f_{\eps_n})\in
A(f_{\eps_n})$. Since $f_{\eps_n}\notin \Pc^\circ$, then
$\mathrm{ca}(f_{\eps_n})$ is the only critical point in
$A(f_{\eps_n})$. A Riemann map $\vp:A(f_{\eps_n})\to \disk$ with
$\vp(0)=0$ conjugates $f_{\eps_n}|_{A(f_{\eps_n})}$ with a normalized
quadratic Blaschke product $Q_{a_n}$, where $a_n\in
\disk$. Then $\vp(\mathrm{ca}(f_{\eps_n}))=c_{a_n}$ is the
unique critical point of $Q_{a_n}$ in $\disk$.
This yields the following contradiction.

\begin{enumerate}

\item[(i)] By Lemma~\ref{l:bpcror},  the point $Q^{m_f}_{a_n}(c_{a_n})$ approaches the unit circle as $\eps_n\to 0$.

\item[(ii)] By Corollary~\ref{c:perturb2} and by continuity, the point $Q^{m_f}_{a_n}(c_{a_n})$
is bounded away from the unit circle
as $\eps_n\to 0$.
\end{enumerate}

A more detailed proof follows.

(i)  Clearly, the multiplier
$(1-\eps_n)\la$ of $f_{\eps_n}$ at $0$ converges to $\la=e^{2\pi i
\ta}$. It follows that the multiplier of $Q_{a_{\eps_n}}$ at $0$ also
converges to $\la$. By Lemma~\ref{l:bpcror}, the point
$Q^{m_f}_{a_n}(c_{a_n})$ approaches the unit circle as
$\eps_n\to 0$.

(ii) On the other hand, take a polynomial $f_\eps$ with small $\eps>0$.
By Corollary~\ref{c:perturb2}, we have $\ol U\subset A(f_\eps)$ for all
sufficiently small $\eps>0$. By continuity,
$f^{m_f}_\eps(\mathrm{ca}(f_\eps))\in K$ if $\eps>0$ is sufficiently
small. Thus, the point $f^{m_f}_{\eps_n}(\mathrm{ca}(f_{\eps_n}))$ is
separated from $\bd(A(f_{\eps_n}))$ by the annulus $U\sm K$ of a
definite positive modulus. It follows, by the conformal invariance of
the modulus, that the point $Q^{m_f}_{a_{\eps_n}}(c_{a_n})$ must
also be separated from $\uc$ by an annulus of a definite positive
modulus. However, this contradicts Theorem~\ref{t:annuprop} and the
conclusions of (i) above.
\end{proof}

Recall (Definition~\ref{d:princrit}) that for an IA-capture %attracting capture
polynomial $f$ we denote by $\om_1(f)$ its critical point that belongs
to $A(f)$ and by $\om_2(f)$ its critical point that does not belong to
$A(f)$ but eventually (after one or more iterations) maps into $A(f)$.
Observe that our notation for critical points $\om_1(f)$ and $\om_2(f)$ is consistent with Definition~\ref{d:princrit}.
Finally, recall that by potentially renormalizable polynomials we mean polynomials in $\Fc$ that do not belong to $\Pc=\ol{\Pc^\circ}$.

\begin{cor}
 \label{c:cri-id}
 Suppose that $f$ is an IS-capture polynomial.
 If $f$ is potentially renormalizable, then $\om_1(f)=\mathrm{re}(f)$ and $\om_2(f)=\mathrm{ca}(f)$.
\end{cor}

\begin{proof}
Since $f$ is potentially renormalizable, all
maps
$f_\eps$ of $f$ are outside $\Pc^\circ$ if $\eps$ is small. By definition
and Lemma \ref{l:zeta-out}, $\mathrm{re}(f)=\om_1(f)$ and $\mathrm{ca}(f)=\om_1(f)$.
\end{proof}

According to Definition~\ref{d:princrit},
a polynomial $f\in \Fc_{at}$ is said to be an \emph{IA-capture (polynomial)} (IA from ``Invariant Attracting'')
if it has a critical point that does not belong to the basin of immediate attraction $A(f)$ of $0$
but  maps to $A(f)$ under a finite iteration of $f$.
A hyperbolic component in $\Fc$ is said to be an \emph{IA-capture
component} if it contains an IA-capture polynomial;
it is easy to see that then all polynomials in it are IA-capture polynomials.
It is well known that all polynomials in an IA-capture component are topologically conjugate on their Julia sets.
Moreover, if $\Wc$ is a hyperbolic component non-disjoint from
$\Fc_{at}$ such that polynomials in $\Wc$ have a critical point which maps
into a cycle of attracting Fatou domains but does not belong to it,
then $\Wc\subset \Fc_{at}$ is an IA-capture component consisting of polynomials $f$ with
an \emph{invariant} attracting Fatou domain $A(f)\ni 0$, a well-defined
critical point $\omega_1(f)\in A(f)$ and a well-defined critical point
$\om_2(f)=\mathrm{ca}(f)\notin A(f)$ such that for some minimal $m_f>0$
we have $f^{m_f}(\om_2(f))\in A(f)$.

\begin{thm}
  \label{t:1hypcomb}
  If $f\in\Fc_{nr}$ is an IS-capture polynomial, then $f$ belongs to the boundary of exactly
  one bounded hyperbolic component $\Wc$ in $\Fc_{at}$.
  Every polynomial $g\in \Wc$ has a locally connected Julia set so that
  $\lam_g=\ol{\lam^r_g}$, and $\Wc$ is either $\Pc^\circ$, or an IA-capture component.
\end{thm}

\begin{proof}
First we consider maps $f_\eps$.
By Lemma~\ref{l:zeta-out}, for some $\delta>0$ and any
$\eps>0$ with $\eps<\delta$, we have $\mathrm{re}(f)\in A(f_\eps)$.
By Corollary \ref{c:perturb2} and continuity, $f^{m_f}(\mathrm{ca}(f_\eps))\in A(f_\eps)$.
Thus, $f_\eps$ is hyperbolic, and there is a unique
hyperbolic component $\Uc$ of $\Fc$ containing all polynomials $f_\eps$ with $\eps<\delta$.
Clearly, $\Uc$ is either $\Pc^\circ$, or an IA-capture component.

By way of contradiction, assume now that $\Uc$ and $\Vc$ are different
bounded hyperbolic components in $\Fc_{at}$ whose boundaries contain
$f$. All polynomials in $\Uc$ have locally connected Julia sets, are
conjugate on their Julia sets, and give rise to the same cubic
invariant lamination $\lam_{\Uc}$; similarly, all polynomials in $\Vc$
give rise to the same cubic lamination $\lam_{\Vc}$ (cf. Corollary
\ref{c:samecomp}).
Since, for a hyperbolic polynomial, the iterated forward images of a critical point cannot lie on the boundary of a
Fatou component, then, by Lemma \ref{l:ratiofull}, we have
$\lam_{\Uc}=\ol{\lam^r_{\Uc}}$ and $\lam_{\Vc}=\ol{\lam^r_{\Vc}}$ where
$\lam^r_{\Uc}$ and $\lam^r_{\Vc}$ are the corresponding rational
laminations.

Consider a leaf $\ell\in\lam^r_f$. It corresponds to a (pre)periodic
point in $J(f)$. Since all periodic points in $J(f)$ are repelling,
then, by Lemma \ref{l:rep}, we have $\ell\in\lam_{\Uc}$ and
$\ell\in\lam_{\Vc}$. Since this holds for any $\ell\in\lam^r_f$, we
conclude that $\lam^r_f\subset \lam^r_{\Uc}$ and
$\lam^r_f\subset\lam^r_{\Vc}$. Now consider a leaf
$\ol{\alpha\beta}\in\lam^r_{\Uc}$. Then $R_g(\alpha)$, $R_g(\beta)$
land at the same (pre)periodic point $x_g$, for every $g\in\Uc$. The
periodic cycle, into which the point $x_g$ eventually maps, is repelling.
Consider a sequence $g_n\in\Uc$ converging to $f$. By Corollary
\ref{c:converge} applied to this sequence, we have
$\ol{\alpha\beta}\in\lam^r_f$. Since $\ol{\alpha\beta}$ is an arbitrary
leaf of $\lam^r_{\Uc}$, we conclude that $\lam^r_\Uc\subset \lam^r_f$.
Similarly, $\lam^r_\Vc\subset\lam^r_f$. Together with the opposite
inclusions proved earlier, this implies that
$\lam^r_\Uc=\lam^r_\Vc=\lam^r_f$. By the first paragraph, it follows
that $\lam_\Uc=\lam_\Vc$. Finally, by Corollary \ref{c:samecomp}, we
have $\Uc=\Vc=\Wc$.
\end{proof}

\begin{proof}[Proof of Theorem A]
Let $f\in\Fc_\la$ be an IS-capture polynomial. By Theorem
\ref{t:1hypcomb}, there is a unique bounded hyperbolic component $\Uc$
in $\Fc_{at}$ with $f\in\bd(\Uc)$. A priori, there could exist a
different hyperbolic component $\Vc$ outside of $\Fc_{at}$ with
$f\in\bd(\Vc)$. Since for $g\in\Vc$ the fixed point $0$ is repelling,
there is a periodic angle $\theta$ such that $R_g(\theta)$ lands at $0$
for all $g\in\Vc$. Consider a sequence $g_n\in\Vc$ converging to $f$.
By Lemma \ref{l:land-ratio}, the ray $R_f(\theta)$ lands at a periodic
point $y\ne 0$ (recall that $0$ is a Siegel point). By Lemma
\ref{l:rep}, the point $y$ is parabolic. However, an IS-capture has no
parabolic periodic points, a contradiction. Thus, $\Uc$ is the only
bounded hyperbolic component in $\Fc$ containing $f$ in
its boundary. It remains to observe that, if $\Uc$ is an IA-capture,
then, by Corollary \ref{c:attracapt}, the polynomial $f$ has a repelling
periodic cutpoint in its Julia set.
\end{proof}

\section{Existence of IS-capture components}
In this section, we find IS-capture components on the boundary of $\Pc^\circ$ as well as on the boundaries of IA-capture components.
Thus we will prove Theorem B.

Let $\Uc$ be an IA-capture component in $\Fc$.
Then, for every $f\in\Uc$, we write $A(f)$ for the immediate attracting basin of $0$.
There is a unique critical point $\om_2(f)$ not in $A(f)$, and we have $f^{m_f}(\om_2(f))\in A(f)$ for some positive integer $m_f$.
We may assume that $m_f$ is the smallest positive integer with this property.
Observe that $m_f$ does not depend on $f$; it depends only on $\Uc$.
We call this integer the \emph{preperiod} of $\Uc$.

\begin{lem}
  \label{l:ag}
  Let $\Uc$ be a hyperbolic component in $\Fc$ that is either $\Pc^\circ$ or an IA-capture component.
  In the latter case, let $m$ be the preperiod of $\Uc$; in the former case, set $m=2$.
  For every Brjuno $\theta\in \R/\Z$ and every $n\ge m$, there exists
  a map $f\in\bd(\Uc)\cap\Fc_\lambda$, where $\lambda=e^{2\pi i\theta}$ and $f^n(c)=0$ for some critical point $c$ of $f$.
  Additionally, it can be arranged that $f^k(c)\ne 0$ for $k<n$.
\end{lem}

%In order to prove Lemma \ref{l:ag}, we use some basic properties of algebraic curves, see e.g. \cite{GH78}.
Let $\Xc_n$ be the set of all polynomials $f\in\Fc$ such that $f^n(c)=0$ for some critical point $c$ of $f$,
 and $n$ is the smallest non-negative integer with this property.
It is clear that $\Xc_n$ is a complex algebraic curve in $\Fc=\C^2$.
%by a polynomial equation and a polynomial inequality, i.e.,
% there are polynomials $P$ and $Q$ on $\Fc$ such that $\Xc_n=\{P=0,\ Q\ne 0\}$.
%More precisely, if we write $f\in\Fc$ as $f_{\la,b}$ and use $(\la,b)$ as coordinates in $\Fc$,
% then $P(\la,b)$ is the resultant of $f^n(c)$ and $f'(c)$ with respect to $c$.
%The polynomial $Q(\la,b)$ is the product of the resultants of $f^k(c)$ and $f'(c)$ over all $k<n$.
Define a function $\mu$ on $\Xc_n$ as $\mu(f)=f'(0)$.

\begin{lem}
  \label{l:center}
Let $\Uc$ be an IA-capture component.
Consider a slice $\Fc_\la$ with $\la\ne 0$ such that $\Fc_\la\cap\Uc\ne\0$; then clearly $|\la|<1$.
Take any integer $n\ge m$, where $m$ is preperiod of $\Uc$.
There is a polynomial $f_!\in\Fc_\la\cap\Uc$ such that $f_!^n(c_!)=0$ for some critical point $c_!$ of $f_!$,
 and $f_!^k(c_!)\ne 0$ for $k<n$.
\end{lem}

\begin{proof}
  The proof is a standard qc-deformation argument, cf. \cite{bf14}.
Take any $f\in\Fc_\la\cap\Uc$.
Then there is a critical point $c$ of $f$ with $f^m(c)\in A(f)$.
The point $v=f(c)$ is contained in a strictly preperiodic Fatou component $V$ of $f$ such that $f^{m-1}(V)=A(f)$.
Consider a $C^1$-homeomorphism $h:\C\to\C$ that coincides with the identity outside of some compact subset of $V$.
Taking iterated $h\circ f$-pullbacks of the standard complex structure in iterated pullbacks of $V$,
 we obtain an $h\circ f$-invariant complex structure on $\C$ that coincides with the standard one outside of iterated pullbacks of $V$.
By the Measurable Riemann Mapping theorem, $h\circ f$ is conjugate to a rational function $f_h$ by a qc-conjugacy fixing $\infty$.
Since $\infty$ is a fixed critical point of $f_h$ of multiplicity 2, we conclude that $f_h$ is a polynomial.
We may also arrange that $f_h\in\Fc$ by an affine change of variables.
In a small neighborhood of $0$, we have $h\circ f=f$, and $f$ is conformally conjugate to $f_h$.
Therefore, $f$ and $f_h$ have the same multiplier at $0$, and $f_h\in\Fc_\lambda$.
Note that $f_h$ depends continuously on $h$, and $f_h=f$ for $h=id$.
Thus any connected set of homeomorphisms $h$ gives rise to a connected subset of $\Fc_\la$ lying entirely in $\Uc$.

We now consider a connected set $\Hc$ of homeomorphisms as above
 (i.e., all $h\in\Hc$ equal the identity outside of some compact subset of $V$).
Let $\Dc$ be the corresponding set of maps $f_h$, where $h$ runs through $\Hc$.
Clearly, $\Dc$ is connected.
For $g=f_h\in\Dc$, define $v_g$ as the image of $h(v)$ under the conjugacy between $h\circ f$ and $f_h$.
Then $v_g$ is a critical value of $g$.
We can choose a homeomorphism $h_!$ so that $f^{n-1}(h_!(v))=0$ and that $f^{k-1}(h_!(v))\ne 0$ for $k<n$.
Moreover, we can arrange that $f^{m-1}(h_!(v))$ is any given $f^{n-m}$-preimage of $0$ in $A(f)$.
This chosen homeomorphism $h_!$ can be included into a connected set $\Hc$ of homeomorphisms.
The corresponding polynomial $f_!=f_{h_!}$ has a critical point $c_!$ corresponding to the critical point $c$ of $h_!\circ f$.
Set $v_!=f_!(c_!)$ to be the corresponding critical value; clearly, it corresponds to the critical value $h_!(v)$ of $h_!\circ f$.
We have $f_!^{n}(c_!)=0$ and $f_!^k(c_!)\ne 0$ for $k<n$.
On the other hand, $f_!$ belongs to a connected set $\Dc$ of hyperbolic polynomials; therefore, $f_!\in\Fc_\la\cap\Uc$.
\end{proof}

The component $\Pc^\circ$ has been extensively studied in \cite{PT09}.
In particular, the following is an immediate corollary of the parameterization of $\Pc^\circ$ obtained in \cite{PT09}:

\begin{lem}
  \label{l:center0}
  Let $\la$ be any complex number with $|\la|<1$, and $n$ be any integer that is at least 2.
Then $\Pc^\circ\cap\Fc_\la$ contains a polynomial $f_!$ with the following properties:
 $f_!^n(c_!)=0$ for some critical point $c_!$ of $f_!$, and $f_!^k(c_!)\ne 0$ for $k<n$.
\end{lem}

Thus, both in the case $\Uc=\Pc^\circ$ and in the case where $\Uc$ is an IA-capture component, we found a certain map $f_!\in\Uc$.

\begin{proof}[Proof of Lemma \ref{l:ag}]
Recall that the function $\mu:\Xc_n\to \C$ was defined by the formula $\mu(f)=f'(0)$.
We claim that $\mu(\Xc_n\cap\Uc)$ coincides with $\disk$, possibly with finitely many punctures.
In the case $\Uc=\Pc^\circ$, this follows from Lemma \ref{l:center0}.
Thus it suffices to assume that $\Uc$ is an IA-capture component.
The inclusion $\mu(\Xc_n\cap\Uc)\subset\disk$ is obvious.
It now suffices to show that $\mu(\Xc_n\cap\Uc)$ is open and closed in $\disk$.
It is open by the Open Mapping Theorem and since $\mu$ is holomorphic.
Suppose now that $\la$ belongs to the boundary of $\mu(\Xc_n\cap\Uc)$ in $\disk$ but not to $\mu(\Xc_n\cap\Uc)$.
Then there is a polynomial $f\in\Fc_\la\cap\ol{\Xc_n\cap\Uc}$.
In other words, there is a sequence $f_i\in\Xc_n\cap\Uc$ with $f_i\to f\in\Fc_\la$ as $i\to\infty$.
For every $i$, there is a critical point $c_i$ of $f_i$ with $f_i^n(c_i)=0$.
Passing to a subsequence, we may assume that $c_i\to c$ as $i\to\infty$, where $c$ is a critical point of $f$, and $f^n(c)=0$.
On the other hand, $|\la|<1$, hence $f$ is hyperbolic.
A hyperbolic polynomial belongs to the closure of a hyperbolic component $\Uc$ only if it belongs to $\Uc$.
Therefore, $f\in\Uc$, but then by definition we have $f\in\Xc_n\cap\Uc$ unless $f$ is a puncture of $\Xc_n$
(which means that $f^k(c)=0$ for some $k<n$).
The latter case is ruled out for the following reason.
By \cite{BuPe}, there is $\delta>0$ such that the basin $A(f_i)$ for all large $i$ contains a $\delta$-disk around $0$.
This implies that $f^k(c)\ne 0$ for $k<n$.
It follows that $\mu(f)$ as $f$ runs through $\ol\Xc_n$ takes all values in $\uc$,
in particular, all values of the form $e^{2\pi i \theta}$, where $\theta$ is Brjuno.

Choose a point $f\in\ol\Xc_n\cap\Uc$ with $\mu(f)=e^{2\pi i\theta}$, where $\theta$ is Brjuno.
It is clear that $f$ is on the boundary of $\Uc$.
We will now prove that $f$ is IS-capture.
Indeed, $f'(0)=\lambda=e^{2\pi i\theta}$ and $\theta$ is Brjuno, hence $f$ has a Siegel disk $\Delta$ around $0$
(we distinguish between the function $\mu$ and its particular value $\lambda$).
On the other hand, since $f\in\ol\Xc_n$, there is a critical point $c$ of $f$ such that $f^{n}(c)=0$.
We have in fact $f\in\Xc_n$ (and $f^k(c)\ne 0$ for $k<n$) for the same reason as above.
By definition, this means that $f$ is an IS-capture polynomial.
\end{proof}

The following statement is proved as Theorem 5.3 in \cite{Z} for a different parameterization of basically the same slices.
The only difference with \cite{Z} is that Zakeri considers critically marked cubic polynomials.

\begin{lem}
  \label{l:no-isol}
  Suppose that $f\in\Fc_\la$, where $|\la|=1$, and $f$ has a Siegel disk $\Delta$ around $0$.
  If $f^n(c)\in\Delta$ for some critical point $c$ of $f$, then there is an IS-capture component in $\Fc_\la$ containing $f$.
\end{lem}

\begin{proof}
The proof is based on the same qc-deformation argument as the proof of Lemma \ref{l:center}.
We will use the notation introduced in Lemma \ref{l:center}, in particular, $v$, $\Hc$, $\Dc$ and $f_h$.
Then $\Dc=\{f_h\,|h\in\Hc\}$ a connected subset of $\Fc_\la$ consisting of IS-capture polynomials.
Recall that $v=f(c)$ is a critical value of $f$.
We choose the set $\Hc$ of homeomorphisms so that $D=\{h(v)\,|\, h\in\Hc\}$ is open.

For every $g\in\Dc$, we let $\Delta_g$ be the Siegel disk of $g$ around $0$.
We let $V_g$ denote the Fatou component of $g$ containing a critical value and such that $g^{n-1}(V_g)=\Delta_g$.
These properties define $V_g$ in a unique way.
We will also write $v_g$ for the critical value of $g$ contained in $V_g$.
Note that, if $g=f_h$, then $v_g$ is the image of $h(v)$ under the conjugacy between $h\circ f$ and $f_h$.
Consider the Riemann map $\phi_g:\Delta_g\to\disk$ such that $\phi(0)=0$ and $\phi'(0)\in\R_{>0}$.
The map $g\mapsto \phi_g(g^{n-1}(v_g))$ takes $\Dc$ to the open set $\phi_f(f^{n-1}(D))$.
Indeed, the image of $f_h$ under this map is $\phi_f(f^{n-1}(h(v)))$.
Thus, an analytic map takes $\Dc$ to some open set.
It follows that $\Dc$ contains an open subset of $\Fc_\la$.
Since $\Dc$ consists of IS-capture polynomials, it is contained in some IS-capture component.
\end{proof}

Finally, we can prove the main theorem of this section.

\begin{thm}
  \label{t:IS-BdP}
Let $\Uc$ be a hyperbolic component of $\Fc$ that is either $\Pc^\circ$ or an IA-capture component.
In the latter case, set $m$ to be the preperiod of $\Uc$; in the former case set $m=2$.
For every Bjuno $\theta\in\R/\Z$ and every $n\ge m$,
  there exists an IS-capture component $\Dc$ in $\bd(\Uc)\cap\Fc_\la$ with $\la=e^{2\pi i\ta}$ such that,
  for all $g\in\Dc$, we have $g^n(c_g)\in\Delta_g$ for some critical point $c_g$ of $g$.
Here $\Delta_g$ is the Siegel disk of $g$ around $0$.
\end{thm}

\begin{proof}
By Lemma \ref{l:ag}, for any Brjuno $\ta\in\R/\Z$ and any $n\ge m$,
 there is a cubic polynomial $f$ with the following properties:
 \begin{enumerate}
   \item we have $f\in\Fc_\la$, where $\la=e^{2\pi i\ta}$;
   \item there is a critical point $c$ of $f$ with $f^n(c)=0$;
   \item we have $f^k(c)\ne 0$ for $k<n$.
 \end{enumerate}
By Lemma \ref{l:no-isol}, there is an IS-capture component $\Dc$ in $\Fc_\la$ containing $f$.
By Theorem A, the component $\Dc$ belongs to the boundary of a unique hyperbolic component $\Vc$ of $\Fc$.
Moreover, by Theorem \ref{t:1hypcomb}, the polynomial $f$ lies on the boundary of a unique hyperbolic component.
But $f$ is in the boundary of $\Uc$.
It follows that $\Vc=\Uc$, hence $\Dc$ is contained in the boundary of $\Uc$.
\end{proof}

Theorem \ref{t:IS-BdP} establishes the existence of many analytic disks on the boundary of the cubic connectedness locus.
Observe that Lemma \ref{l:no-isol} and Theorem \ref{t:IS-BdP} imply Theorem B.

We conclude this section with a remark which relates our results concerning IA-capture components and laminations.
A cubic invariant lamination $\lam$ is said to be an \emph{IA-capture lamination} if the following assumptions hold:
\begin{enumerate}
  \item there is an invariant Fatou gap $A$ such that $\si_3|_{A\cap\uc}$ is two-to-one;
  \item there is a Fatou gap $V\ne A$ such that $\si_3|_{V\cap\uc}$ is two-to-one;
  \item we have $\si_3^{m_\lam}(V)=A$, where $m_\lam=m$ is the minimal integer with this property.
\end{enumerate}
The number $m$ is called the \emph{preperiod of $\lam$}.
It is well-known (and easy to see) that any IA-lamination is the closure of its restriction upon all the rational angles
(i.e., the closure of the corresponding rational lamination).

It follows from the appendix to \cite{mil12} written by Poirier that, for each IA-capture lamination $\lam$,
 there exists a unique IA-capture component $\Uc_\lam\subset \Fc$ with the following property.
No matter which $f\in \Uc_\lam$ we take, the lamination generated by $f$ coincides with $\lam$.
The result of \cite{mil12} is stated in the language of Hubbard trees and so-called reduced mapping schemes,
 however, a straightforward translation of this result into the language of laminations yields the claim stated above.
Similarly, if $\lam$ is the empty lamination, then we set $\Uc_\lam=\Pc^\circ$.
Evidently, Theorem \ref{t:IS-BdP} can be restated to emphasize the role of IA-capture laminations, e.g., as follows.

\medskip

\noindent{\textbf{Theorem \ref{t:IS-BdP}$'$.}}
\emph{Let $\lam$ be the empty lamination or an IA-capture lamination.
In the latter case, set $m$ to be the preperiod of $\lam$; in the former case set $m=2$.
For every Brjuno $\theta\in\R/\Z$ and every $n\ge m$, the hyperbolic component $\Uc_\lam$ contains an IS-capture component $\Dc$ in $\bd(\Uc_\lam)\cap\Fc_\la$ with $\la=e^{2\pi i\ta}$ such that, for all
$g\in\Dc$, we have $g^{\circ n}(c_g)\in\Delta_g$ for some critical
point $c_g$ of $g$, and $n$ is the least such integer. Here $\Delta_g$
is the Siegel disk of $g$ around $0$.}

\section{The main cubioid of $\Fc$}
\label{s:cu}
In this section, we prove Theorem C and
obtain corollaries related to the problem of distinguishing between Siegel and Cremer fixed points.
Recall that the Main Cubioid $\mathcal{CU}$ was introduced in Definition~\ref{d:cubio}.

\begin{proof}[Proof of Theorem C]
Suppose that $f\in\Fc_\la$ with $|\la|=1$ is a cubic IS-capture
polynomial. By way of contradiction, assume that
$f\in\mathcal{CU}\sm\Pc$. By Theorem \ref{t:1hypcomb}, all polynomials
$f_\eps$ for small $\eps>0$ belong to some IA-capture component $\Uc$
(since $f\notin \Pc$, we have $f_\eps\notin \Pc^\mathrm{o}$ for small
$\eps$). On the other hand, then, by Theorem A, the map $f$ contains a
repelling periodic cutpoint in its Julia set, a contradiction with
$f\in \mathcal{CU}$. The rest of Theorem C now follows from
Theorem~\ref{t:sie-quee}.
\end{proof}

For the sake of completeness we also prove the next lemma.

\begin{lem}
\label{l:nohyp}
  The only hyperbolic component of $\Fc$ intersecting $\mathcal{CU}$ is $\Pc^\circ$.
\end{lem}

\begin{proof}
Assume, to the contrary, that there exists a hyperbolic component $\Vc\ne \Pc^\circ$ intersecting $\mathcal{CU}$.
Set $\Vc_\la=\Vc\cap\Fc_\la$ and $\mathcal{CU}_\la=\mathcal{CU}\cap\Fc_\la$.
Choose $\la$ with $\Vc_\la\cap\mathcal{CU}_\la\ne\0$.
We must have $|\la|\le 1$ since otherwise $\mathcal{CU}_\la=\0$.
%Clearly, $\Vc$ is disjoint from $\Pc$ (otherwise we would have $\Vc\cap\Pc^\circ\ne\0$).
\begin{comment}
Then $\Vc\subset\mathcal{CU}\sm\Pc$.
Indeed, by definition of $\mathcal{CU}$, if some hyperbolic polynomial belongs to $\mathcal{CU}$,
then all polynomials from the same hyperbolic component also belong to $\mathcal{CU}$.
\end{comment}
From $\Vc_\la\ne\0$, it follows that $\Vc\cap\Fc_{at}\ne\0$.
But then $\Vc\subset\Fc_{at}$ and $|\la|<1$.
Note also that, since polynomials in $\mathcal{CU}$ have connected Julia sets,
  all polynomials in $\Vc$ have connected Julia sets, i.e., the component $\Vc$ is bounded.

Take $g\in\Vc_\la\cap\mathcal{CU}_\la$.
Then $J(g)$ is locally connected; let $\lam$ be the corresponding geodesic lamination.
There is a gap $G$ of $\lam$ corresponding to $A(g)$.
By Theorem \ref{t:major}, the major $M$ of $G$ is either critical or periodic.
The former implies that a critical point of $g$ belongs to $\bd(A(g))$, a contradiction.
Therefore, $M=\ol{\al\be}$ is periodic.
The rays $R_g(\al)$, $R_g(\be)$ land at the same periodic point $x$ of $g$.
Since $g$ is hyperbolic, $x$ must be repelling.
Thus $g$ has a repelling periodic cutpoint of $J(g)$, a contradiction with $g\in\mathcal{CU}$.
\end{proof}

A question as to whether a fixed irrationally indifferent point of a
polynomial is Cremer or Siegel depending on the multiplier at this
point is addressed in a conjecture by A. Douady. Let us now state a
related to it corollary based upon results of Perez-Marco.

%For brevity we discuss this issue only for polynomials.

\begin{dfn}[Brjuno numbers]\label{d:bn}
The set $\mathcal{B}$ is the set of irrational numbers $\theta$ such
that $\sum \frac{\ln{q_{n+1}}}{q_n} < \infty$, where
$\frac{p_n}{q_n}\to \theta$ is the sequence of approximations given by
the continued fraction expansion of $\theta$. Numbers from
$\mathcal{B}$ are called \emph{Brjuno} numbers.
\end{dfn}

The following is a classical result by A. D. Brjuno \cite{brj71}.

\begin{thm}[\cite{brj71}]\label{t:br}
If $a$ is an irrationally indifferent fixed point of a polynomial $f$ with multiplier $e^{2\pi i \theta}$ and $\theta\in \mathcal{B}$, then
the point $a$ is a Siegel fixed point.
\end{thm}

Another classical result, due to J.-C. Yoccoz, states that %in fact
in the quadratic case Theorem~\ref{t:br} is sharp.

\begin{thm}[\cite{yoc95}]\label{t:yoc} In the situation of Theorem~\ref{t:br},
if $f$ is quadratic and $\theta\notin \mathcal{B}$ is not a Brjuno
number, then $a$ is a Cremer fixed point of $f$.
\end{thm}

A conjecture by A. Douady states that %in fact
Theorem~\ref{t:yoc} holds
for higher degree polynomials too. Below we verify this for cubic
polynomials $f_{\lambda, b}=\lambda z + bz^2+z^3$ except for
polynomials that belong to the set $\Pc_\lambda$. An important
ingredient of our arguments is a result of R. Perez-Marco \cite{per01};
again for brevity we state only a relevant corollary of Perez-Marco's
theorem reduced to our spaces of polynomials (the actual results of
\cite{per01} are much stronger and much more general).

\begin{cor}[Corollary 1 \cite{per01}]\label{c:pm}
Suppose that $\lambda=e^{2\pi i \theta}$ and $\theta$ is irrational.
Then the set of parameters $b$ for which $f_{\lambda, b}$ has $0$ as a
Siegel fixed point is either the entire $\Fc_\lambda$, or, otherwise,
has Hausdorff dimension $0$ $($in particular, it has empty interior$)$.
\end{cor}

Combining these results with our tools we prove
Corollary~\ref{c:nop}.

\begin{cor}\label{c:nop}
If $\theta\notin \mathcal{B}$ is not a Brjuno number and
$\lambda=e^{2\pi i \theta}$, then the fact that $f\in \Fc_\lambda\sm
\Pc_\lambda$ implies that $0$ is a Cremer fixed point of $f$.
\end{cor}

\begin{proof}
Suppose first that $f=f_{\lambda, b}\notin \thu({\Pc}_\lambda)$. Then, by
\cite{bopt16b}, the map is immediately renormalizable; moreover, $0$
belongs to the filled quadratic-like Julia set $K^*\subset K(f)$ of $f$.
By Theorem \ref{t:yoc}, this implies that $0$ is a Cremer point of $f$ .
By Corollary~\ref{c:pm}, it follows then that the set of
parameters $b$ for which $f_{\lambda, b}$ has $0$ as a Siegel point has
empty interior. Since, by \cite{bopt16b}, in each component of
$\thu({\Pc}_\lambda)\sm {\Pc}_\lambda$ the polynomials are conjugate,
then polynomials in those bounded domains cannot have $0$ as their
fixed Siegel point. This completes the proof.
\end{proof}

\section*{Acknowledgements}

The first author was partially supported by NSF grant DMS--1201450. The
third named author was partially supported by NSF grant DMS--1807558.
The fourth named author was partially funded within the framework of the HSE University Basic Research Program and
by the Russian Academic Excellence Project `5-100'.

\end{document}